\documentclass[a4paper,11pt]{amsart}
\usepackage[titletoc,toc,title]{appendix}
\usepackage[left=2.7cm,right=2.7cm,top=3.5cm,bottom=3cm]{geometry}
\usepackage{amssymb,latexsym,amsmath,amsthm,amscd}
\usepackage{graphicx}
\usepackage{dsfont}
\usepackage[all]{xy}
\usepackage{mathrsfs}
\usepackage{color}
\definecolor{Blue}{rgb}{0.3,0.3,0.9}

\usepackage[T2A,OT1]{fontenc}

\newcommand{\sk}{\vspace{0.1in}}
\newtheorem{thm}{Theorem}[section]
\newtheorem{def-thm}[thm]{Definition-Theorem}
\newtheorem{cor}[thm]{Corollary}
\newtheorem{lem}[thm]{Lemma}
\newtheorem{def-lem}[thm]{Definition-Lemma}
\newtheorem{prop}[thm]{Proposition}
\newtheorem{conj}[thm]{Conjecture}
\newtheorem*{ThmA}{Theorem A}
\newtheorem*{ThmB}{Theorem B}
\theoremstyle{definition}
\newtheorem{defn}[thm]{Definition}
\theoremstyle{remark}
\newtheorem{rem}[thm]{Remark}
\numberwithin{thm}{section}
\numberwithin{equation}{section}

\newcommand{\cO}{\mathcal{O}}

\newcommand{\frakm}{\mathfrak{m}}

\newcommand{\frakl}{\mathfrak{l}}
\newcommand{\pp}{\mathfrak{p}}
\newcommand{\frakp}{\mathfrak{p}}
\newcommand{\frakq}{\mathfrak{q}}

\newcommand{\frakf}{\mathfrak{c}}

\newcommand{\fraka}{\mathfrak{a}}
\newcommand{\frakc}{\mathfrak{c}}

\newcommand{\cI}{\mathcal{I}}
\newcommand{\bQ}{\mathbf{Q}}
\newcommand{\bZ}{\mathbf{Z}}
\newcommand{\bC}{\mathbf{C}}

\newcommand{\ro}{\mathbf{Z}_p}

\def\ac{{\rm ac}}
\def\cyc{{\rm cyc}}

\newcommand{\Ac}{{\mathbf{A}^{\rm ac}}}
\newcommand{\Tc}{{\mathbf{T}^{\rm ac}}}

\newcommand{\Gr}{}
\newcommand{\can}{\Psi}
\newcommand{\unr}{R_0}

\newcommand{\tors}{\rm tors}

\begin{document}

\title[$p$-adic heights of Heegner points and Beilinson--Flach elements]
{$p$-adic heights of Heegner points and Beilinson--Flach elements}
\author[F.~Castella]{Francesc Castella}
\address{Mathematics Department, Princeton University, Fine Hall, Princeton,
NJ 08544-1000, USA}
\email{fcabello@math.princeton.edu}

\thanks{This project has received funding from the European Research Council (ERC) under the European Union's
Horizon 2020 research and innovation programme (grant agreement No. 682152).}






\maketitle

\begin{abstract}
We give a new proof of Howard's $\Lambda$-adic Gross--Zagier formula \cite{howard-compmath}, which
we extend to the context of indefinite Shimura curves over $\bQ$ attached to nonsplit quaternion algebras.
This formula relates the cyclotomic derivative of a two-variable $p$-adic $L$-function
restricted to the anticyclotomic line to the cyclotomic $p$-adic heights of Heegner points
over the anticyclotomic tower, and our proof,
rather than inspired by the original approaches of Gross--Zagier \cite{GZ} and
Perrin-Riou \cite{PR89}, is via Iwasawa theory,
based on the connection between Heegner points, Beilinson--Flach elements, and their explicit reciprocity laws.
\end{abstract}

\setcounter{tocdepth}{2}
\tableofcontents

\section*{Introduction}

Let $E/\bQ$ be an  elliptic curve of conductor $N$, let $K/\bQ$ be an imaginary quadratic field of
discriminant $-D_K<0$ prime to $N$, and let $f=\sum_{n=1}^\infty a_n(f)q^n\in S_2(\Gamma_0(N))$ be the
normalized newform associated with $E$. 
The field $K$ determines a factorization $N=N^+N^-$
such that a prime $\ell$ divides $N^+$ if $\ell$ is split in $K$, and divides $N^-$ if $\ell$ is inert in $K$.
Throughout this paper, we shall assume that
\begin{equation}\label{eq:sq-free}
\textrm{$N$ is square-free,}\tag{$\square$-free}
\end{equation}
and that $K$ satisfies the following \emph{generalized Heegner hypothesis} relative to $N$:
\begin{equation}\label{eq:HH}
\textrm{$N^-$ is the product of an \emph{even} number of primes.}\tag{Heeg}
\end{equation}

Under the latter hypothesis, $E$ can be realized as a quotient of the Jacobian ${\rm Jac}(X_{N^+,N^-})$
of a Shimura curve $X_{N^+,N^-}$ attached to the pair $(N^+,N^-)$.
More precisely, after possibly replacing $E$ by a $\bQ$-isogenous elliptic curve, we may fix a parametrization
\[
\Phi_E:{\rm Jac}(X_{N^+,N^-})\longrightarrow E.
\]

Fix also a prime $p\geqslant 5$ of good ordinary reduction for $E$,
let $K_\infty^\cyc$ and $K_\infty^\ac$ be the cyclotomic and anticyclotomic $\bZ_p$-extensions of $K$, respectively,
and set $K_\infty=K_\infty^\ac K_\infty^\cyc$. 
Then, by work of Hida
and Perrin-Riou (see \cite{hidaI}, \cite{PR38}, \cite{PR89}),
there is a $p$-adic $L$-function
\begin{equation}\label{eq:SU}
L_p(f/K)\in\ro[[{\rm Gal}(K_\infty/K)]]\otimes_{\ro}\bQ_p\nonumber
\end{equation}
interpolating the central values for the Rankin--Selberg convolution of $f$
with the theta series attached to finite order characters of ${\rm Gal}(K_\infty/K)$.
(This is denoted $L_p(f/K,\Sigma^{(1)})$ in the body of the paper.)
\sk

Letting $\Lambda^\ac:=\ro[[{\rm Gal}(K_\infty^\ac/K)]]$, we may identify
\[
\ro[[{\rm Gal}(K_\infty/K)]]\simeq\Lambda^\ac[[{\rm Gal}(K_\infty^\cyc/K)]],
\]
and hence upon the choice of a topological generator $\gamma^\cyc\in{\rm Gal}(K^\cyc_\infty/K)$, we may
expand 
\begin{equation}\label{eq:intro-expand}
L_p(f/K)=
L_{p,0}(f/K)+L_{p,1}^{\rm cyc}(f/K)(\gamma^\cyc-1)+\cdots\nonumber
\end{equation}
as a power series in $(\gamma^\cyc-1)$ with coefficients in $\Lambda^\ac\otimes_{\bZ_p}\bQ_p$.
As a consequence of hypothesis (\ref{eq:HH}), the sign in the functional equation satisfied by $L_{p}(f/K)$
forces the vanishing of $L_{p,0}(f/K)$, and we are led to consider
the `linear term' $L_{p,1}^{\rm cyc}(f/K)$ in the above expansion. 
\sk


On the other hand, taking certain linear combinations of Heegner points on $X_{N^+,N^-}$ defined
over ring class fields of $K$ of $p$-power conductor, and
mapping them onto $E$ via $\Phi_E$, it is possible to define a 
class $\mathbf{z}_f$ in the Selmer group
\[
{\rm Sel}(K,\Tc):=\varprojlim_n\varprojlim_m{\rm Sel}_{p^m}(E/K_n^{\ac})
\]
attached to the $E$ over the tower $K_\infty^\ac/K$, where ${\rm Sel}_{p^m}(E/K_n^{\ac})\subseteq
H^1(K_n^{\rm ac},E[p^m])$ is the usual $p^m$-th Selmer group.
Let $T_pE=\varprojlim_mE[p^m]$ be the $p$-adic Tate module of $E$.
In Section~\ref{subsec:rubin} we recall the definition of a cyclotomic $\Lambda^\ac$-adic height pairing
\begin{equation}\label{eq:lambda-ht}
\langle\;,\;\rangle_{K^\ac_\infty}^{\rm cyc}:{\rm Sel}_{\Gr}(K,\Tc)
\times{\rm Sel}_{\Gr}(K,\Tc)\longrightarrow\bQ_p\otimes_{\bZ_p}\Lambda^\ac\otimes_{\bZ_p}\mathcal{I}/\mathcal{I}^2,\nonumber
\end{equation}
where $\mathcal{I}$ is the augmentation ideal of $\ro[[{\rm Gal}(K_\infty^\cyc/K)]]$.
Using our fixed topological generator $\gamma^\cyc$, we may view
$\langle\;,\;\rangle_{K^\ac_\infty}^{\rm cyc}$
as taking values in $\Lambda^\ac\otimes_{\bZ_p}\bQ_p$.
\sk

The main result of this paper is then the following $\Lambda^\ac$-adic Gross--Zagier formula.

\begin{ThmA}\label{thm:GZ}
In addition to {\rm (\ref{eq:HH})} and {\rm (\ref{eq:sq-free})}, assume that:
\begin{itemize}
\item{} $p=\pp\overline{\pp}$ splits in $K$,
\item{} $N^-\neq 1$,
\item{} $E[p]$ is ramified at every prime $\ell\mid N^-$,
\item{} ${\rm Gal}(\overline{\bQ}/K)\rightarrow{\rm Aut}_{\bZ_p}(T_pE)$ is surjective.
\end{itemize}
Then we have the equality
\[
(L_{p,1}^{\rm cyc}(f/K))=
(\langle\mathbf{z}_f,\mathbf{z}_f\rangle_{K^\ac_\infty}^{\rm cyc})
\]
as ideals of $\Lambda^\ac\otimes_{\bZ_p}\bQ_p$.
\end{ThmA}

For $N^-=1$, Theorem~A follows from Howard's extension in the anticyclotomic
direction of Perrin-Riou's $p$-adic Gross--Zagier formula (see \cite{howard-compmath}).
Even though it should be possible to extend Howard's calculations to also cover the cases considered
here\footnote{See \cite{disegni-pGZ} for an approach along these lines.}, 
our proof of Theorem~A is completely different. In fact, as we shall explain in the remaining part of
this Introduction, rather than inspired by the work of Gross--Zagier \cite{GZ} and Perrin-Riou~\cite{PR89}, our proof 
is via Iwasawa theory. A key advantage (and arguably the main interest) of this new alternative approach is that it seems amenable to generalization
to higher weights\footnote{After the first version of this paper was released, this has largely been carried out
 by the combined works of B{\"u}y{\"u}kboduk--Lei \cite{BL-ord} and Longo--Vigni \cite{LV-cycles}.},
to the supersingular case (where an analogue of Howard's $\Lambda$-adic Gross--Zagier formula was
not known\footnote{See the recent preprints \cite{cas-wan-SS} and \cite{BuyLei-SS}.}),
and even to the setting of `big Heegner points' in Hida families \cite{howard-invmath}.
The application of the latter extension in connection to well-known nonvanishing conjectures
due to Howard \cite{howard-invmath} and Greenberg \cite{GreenbergCRM} will be explained in
forthcoming work 
\cite{cas-wan-big}.
\sk

Our starting point is the observation (already pointed out by Howard in the introduction to \cite{howard-compmath})
that the content of Theorem~A can be parlayed into an equality without reference to Heegner points.
Indeed, let $K_n^\ac/K$ be the unique subextension of $K_\infty^\ac/K$ with ${\rm Gal}(K_n^\ac/K)\simeq\bZ/p^n\bZ$, and let
${\rm Sel}_{p^\infty}(E/K_n^\ac)\subseteq H^1(K_n^\ac,E[p^\infty])$ be the usual $p$-power Selmer group.
Under hypothesis \rm{(\ref{eq:HH})}, and as a reflection of the above vanishing of $L_{p,0}(f/K)$,
the Selmer group
\[
{\rm Sel}(K,\Ac):=\varinjlim_n\varinjlim_m{\rm Sel}_{p^m}(E/K_n^\ac)
\]
can be shown to have positive corank over the anticyclotomic Iwasawa algebra $\Lambda^\ac$.
Moreover, the `Heegner point main conjecture' formulated by Perrin-Riou \cite{PR-HP} predicts that both
${\rm Sel}(K,\Tc)$ and the Pontryagin dual $X(K,\Ac):={\rm Sel}(K,\Ac)^\vee$ have $\Lambda^\ac$-rank one, and that
letting the subscript `tors' denote the $\Lambda^\ac$-torsion submodule, we have the equality
\begin{equation}\label{eq:intro-HPMC}
Ch_{\Lambda^\ac}(X(K,\Ac)_{\rm tors})\overset{}=
Ch_{\Lambda^\ac}\biggl(\frac{{\rm Sel}(K,\Tc)}{\Lambda^\ac\cdot\mathbf{z}_f}\biggr)^2\tag{HP}
\end{equation}
as ideals in $\Lambda^\ac\otimes_{\bZ_p}\bQ_p$.
\sk

Building on the work of Cornut--Vatsal \cite{CV-doc} (which implies the nontriviality of $\mathbf{z}_f$)
and an extension to the anticyclotomic setting of Mazur--Rubin's theory of Kolyvagin systems \cite{MR-KS},
Howard established in \cite{howard-PhD-I} one of the divisibilities in (\ref{eq:intro-HPMC}),
while the converse divisibility has been more recently obtained by Wan \cite{wan}. 
Letting $\mathcal{R}_{\rm cyc}$ be the characteristic ideal of the cokernel of
$\langle\;,\;\rangle^{\rm cyc}_{K^\ac_\infty}$, one sees without difficulty that Theorem~A
is a consequence of the following $\Lambda^\ac$-adic analog of the Birch and Swinnerton-Dyer conjecture.

\begin{ThmB}
Under the hypotheses of Theorem~A, we have
\begin{equation}\label{eq:reformulate}
(L_{p,1}^{\rm cyc}(f/K))\overset{}=\mathcal{R}_{\rm cyc}\cdot Ch_{\Lambda^\ac}(X(K,\Ac)_{\rm tors})\nonumber
\end{equation}
as ideals in $\Lambda^\ac\otimes_{\bZ_p}\bQ_p$.
\end{ThmB}

For rational elliptic curves with complex multiplication by $K$,
an analog of Theorem~B was obtained by Agboola--Howard~\cite{AHsplit} using the Euler system
of elliptic units, Rubin's proof of the Iwasawa main conjecture for $K$,
and a variant of Rubin's $p$-adic height formula \cite{rubin-ht}.
By (\ref{eq:sq-free}), 
the CM setting is excluded here,
but nonetheless we are able to prove Theorem~B by a strategy similar to theirs,
using the Euler system of Beilinson--Flach elements developed 
in a remarkable series of works \cite{LLZ,LLZ-K,KLZ2,KLZ1} by Lei--Loeffler--Zerbes and Kings--Loeffler--Zerbes
which refined and generalized earlier work of Bertolini--Darmon--Rotger \cite{BDR1,BDR2}.
\sk

Notwhithstanding this similarity,
a notable difference between the approach in this paper and that in \cite{AHsplit} is that,
even though we essentially rely on the explicit reciprocity laws for Beilinson--Flach elements \cite{KLZ2},
we do not exploit their Euler system properties; 
instead, we build on Howard's results \cite{howard-PhD-I,howard-PhD-II}
on the anticyclotomic Euler system of classical Heegner points. 
This appears to be an important technical point, since it allows us to dispense with the twist of $E$ by a
rather undesirable $p$-distinguished character.
\sk


\noindent\emph{Acknowledgements.}
It is a pleasure to thank David Loeffler and Sarah Zerbes for their interest in this work.
We also thank Victor Rotger and Chris Skinner for enlightening conversations, and
the anonymous referee for a very careful reading of this paper and a number of suggestions
which lead to significant improvements in the exposition.

\section{$p$-adic $L$-functions}\label{sec:padicL}

Throughout this section, we let $f=\sum_{n=1}^\infty a_n(f)q^n\in S_2(\Gamma_0(N_f))$
be a normalized newform of conductor $N_f$, and let $K$ be an imaginary quadratic field of discriminant $-D_K<0$ prime to $N_f$.
We denote by $\cO_K$ the ring of integers of $K$. Fix a prime $p\nmid 6N_fD_K$, and assume that
\begin{equation}\label{eq:spl}
\textrm{$p=\frakp\overline{\pp}$ splits in $K$.}\tag{spl}
\end{equation}
Fix complex and $p$-adic embeddings
$\overline{\bQ}\overset{\imath_\infty}\hookrightarrow\bC$ and
$\overline\bQ\overset{\imath_p}\hookrightarrow\bC_p$, and let $\pp$ be
the prime of $K$ above $p$ induced by $\imath_p$.
Finally, since it will suffice for the applications in this paper, we assume that
the number field generated by the coefficients $a_n(f)$ embeds into $\bQ_p$.

\subsection{Hida's $p$-adic Rankin $L$-series, I}\label{sec:2varL}

Let $\Xi_K$ denote the set of algebraic Hecke characters
$\psi:K^\times\backslash\mathbb{A}_K^\times\rightarrow\bC^\times$.
We say that $\psi\in\Xi_K$ has infinity type $(\ell_1,\ell_2)\in\bZ^2$ if
\[
\psi_\infty(z)=z^{\ell_1}\overline{z}^{\ell_2},
\]
where for each place $v$ of $K$ we let $\psi_v:K_v^\times\rightarrow\bC^\times$ be
the $v$-component of $\psi$. The conductor of $\psi$ is the largest ideal $\frakf\subseteq\cO_K$ such that
$\psi_\frakq(u)=1$ for all $u\in(1+\frakf\cO_{K,\frakq})^\times\subseteq K_\frakq^\times$.
If $\psi$ has conductor $\frakc_\psi$ and
$\fraka$ is any fractional ideal of $K$ prime to $\frakc_\psi$,
we write $\psi(\fraka)$ for $\psi(a)$, where $a$ is any idele satisfying $a\hat\cO_K\cap K=\fraka$ and
$a_\frakq=1$ for all primes $\frakq$ dividing $\frakc_\psi$. As a function on fractional ideals, then $\psi$ satisfies
\[
\psi((\alpha))=\alpha^{-\ell_1}\overline{\alpha}^{-\ell_2}
\]
for all $\alpha\in K^\times$ with $\alpha\equiv 1\pmod{\frakf_\psi}$.


We say that a Hecke character $\psi$ of infinity type $(\ell_1,\ell_2)$ is
\emph{critical} (for $f$) if the value $s=1$ is a critical value in the sense of Deligne
for
\[
L(f/K,\psi,s)=L\left(\pi_f\times\pi_{\psi},s+\frac{\ell_1+\ell_2-1}{2}\right),
\]
where $L(\pi_f\times\pi_{\psi},s)$ is the $L$-function for the Rankin--Selberg
convolution of the automorphic representations of $\mathbf{GL}_2(\mathbb{A})$
corresponding to $f$ and the theta series $\theta_{\psi}$ associated to $\psi$.
The set of infinity types of critical characters can be written as the disjoint union
\[
\Sigma=\Sigma^{(1)}\sqcup\Sigma^{(2)}\sqcup\Sigma^{(2')},
\]
with $\Sigma^{(1)}=\{(0,0)\}$,
$\Sigma^{(2)}=\{(\ell_1,\ell_2)\colon \ell_1\leqslant -1, \ell_2\geqslant 1\}$,
$\Sigma^{(2')}=\{(\ell_1,\ell_2)\colon \ell_2\leqslant -1, \ell_1\geqslant 1\}$.

Denote by $\psi^\rho$ the character obtained by composing $\psi:K^\times\backslash\mathbb{A}_K^\times\rightarrow\bC^\times$
with the complex conjugation on $\mathbb{A}_K^\times$. The involution $\psi\mapsto\psi^\rho$ on $\Xi_K$ has the effect on infinity types
of interchanging the regions $\Sigma^{(2)}$ and $\Sigma^{(2')}$ (while leaving $\Sigma^{(1)}$ stable).
Since the values $L(f/K,\psi,1)$ and $L(f/K,\psi^\rho,1)$ are the same, for the purposes of
$p$-adic interpolation we may restrict our attention
to the first two subsets in the above the decomposition of $\Sigma$. 

\begin{defn}
Let $\psi=\psi^{\infty}\psi_\infty\in\Xi_K$ be an algebraic Hecke character
of infinity type $(\ell_1,\ell_2)$. The \emph{$p$-adic avatar} of $\psi$
is the character $\hat{\psi}:K^\times\backslash\hat{K}^\times\rightarrow\bC_p^\times$
defined by
\[
\hat\psi(z)=\imath_p\imath_\infty^{-1}(\psi^{\infty}(z))z_\pp^{\ell_1}z_{\overline{\pp}}^{\ell_2}.
\]
\end{defn}

For each ideal $\frakm$ of $K$, let $H_\frakm\simeq {\rm Gal}(K(\frakm)/K)$
be the ray class group of $K$ modulo $\frakm$, and set $H_{p^\infty}=\varprojlim_rH_{p^r}$.
By the 
Artin reciprocity map, the correspondence $\psi\mapsto\hat\psi$ establishes a
bijection between the set of algebraic Hecke characters of $K$ of conductor dividing $p^\infty$ and
the set of locally algebraic $\overline{\bQ}_p$-valued characters of $H_{p^\infty}$.

For the next theorem, let $N$ be a positive integer divisible by $D_KN_f$ and with
the same prime factors as $D_KN_f$.

\begin{thm}\label{thm:hida}
Let $\alpha$ and $\beta$ be the roots of $x^2-a_p(f)x+p$, ordered so that $0\leqslant v_p(\alpha)\leqslant v_p(\beta)$.
\begin{itemize}
\item[$(i)$]{}
There exists a $p$-adic $L$-function
$L_p(f/K,\Sigma^{(2)})\in{\rm Frac}(\ro[[H_{p^\infty}]])$
such that for every character $\chi\in\Xi_K$ of trivial conductor
and infinity type $(\ell_1,\ell_2)\in\Sigma^{(2)}$, we have
\begin{align*}
L_p(f/K,\Sigma^{(2)})(\hat\psi)&=\frac{\Gamma(\ell_2)\Gamma(\ell_2+1)\cdot\mathcal{E}(\psi,f)}
{(1-\psi^{1-\rho}(\pp))(1-p^{-1}\psi^{1-\rho}(\pp))}
\cdot\frac{L(f/K,\psi,1)}{(2\pi)^{2\ell_2+1}\cdot\langle\theta_{\psi_{\ell_2}},\theta_{\psi_{\ell_2}}\rangle_{N}},
\end{align*}
where $\theta_{\psi_{\ell_2}}$ is the theta series of weight $\ell_2-\ell_1+1\geqslant 3$ associated to the
Hecke character $\psi_{\ell_2}:=\psi{\rm\mathbf{N}}_K^{\ell_2}$ of infinity type $(\ell_1-\ell_2,0)$, and
\begin{equation}\label{eq:Euler-Hida}
\mathcal{E}(\psi,f)=(1-p^{-1}\psi(\pp)\alpha)(1-p^{-1}\psi(\pp)\beta)
(1-\psi^{-1}(\overline\pp)\alpha^{-1})(1-\psi^{-1}(\overline\pp)\beta^{-1}).\nonumber
\end{equation}
\item[$(ii)$]{}
If $v_p(\alpha)=0$,
there exists a $p$-adic $L$-function
$L_p(f/K,\Sigma^{(1)})\in\bQ_p\otimes_{\bZ_p}\ro[[H_{p^\infty}]]$
such that for every finite order character $\psi\in\Xi_K$
of conductor $\frakp^m\overline{\frakp}^n$, we have
\begin{align*}
L_p(f/K,\Sigma^{(1)})(\hat\psi)&=
\frac{\prod_{\frakq\mid p}(1-p^{-1}\beta\psi(\frakq))
(1-\alpha^{-1}\psi^{-1}(\frakq))}{(1-\beta\alpha^{-1})(1-\beta p^{-1}\alpha^{-1})}
\cdot C(f,\psi)\cdot\frac{L(f/K,\psi,1)}{8\pi^2\cdot\langle f,f\rangle_N},
\end{align*}
where
\[
C(f,\psi)=
\left\{
\begin{array}{ll}
-iN & \textrm{if $m=n=0$;}\\
\tau(\psi)\alpha^{-m-n+1}(Np^{n+m})^{1/2} & \textrm{if $m+n>0$,}
\end{array}
\right.
\]
with $\tau(\psi)$ denoting the root number of $\psi$.
\end{itemize}
\end{thm}

\begin{proof}
The first part is a reformulation of \cite[Thm.~6.1.3(ii)]{LLZ-K}, and the second part follows from
an extension of the calculations in \cite[Thm.~6.1.3(i)]{LLZ-K}) (see e.g. \cite[Thm.~2.1]{BL-ord} for details). 
%
\end{proof}

\begin{rem}\label{rem:non0}
The $p$-adic $L$-functions of Theorem~\ref{thm:hida} are nonzero.
Indeed, as shown in the proof of \cite[Prop.~9.2.1]{KLZ0}, the nonvanishing of $L_p(f/K,\Sigma^{(2)})$
follows immediately from the existence of infinitely many characters $\psi$ with infinity type $(\ell_1,\ell_2)\in\Sigma^{(2)}$
for which the Euler product defining $L(f/K,\psi,s)$ converges at $s=1$, 
while the nonvanishing of $L_p(f/K,\Sigma^{(1)})$ (in fact, even of the ``cyclotomic'' restriction
of $L_p(f/K,\Sigma^{(1)})$ to characters of $H_{p^\infty}$ factoring through the norm map)
follows from the nonvanishing results of \cite{rohrlich}.
\end{rem}


\subsection{Anticyclotomic $p$-adic $L$-functions}\label{sec:anti-L}



We keep the notations from $\S\ref{sec:2varL}$, and write
\[
N_f=N^+N^-
\]
with $N^+$ (resp. $N^-$) equal to the product of the prime factors of $N_f$ split (resp. inert) in $K$.
As in the Introduction, we say that the pair $(f,K)$ satisfies the \emph{generalized Heegner hypothesis} if
\begin{equation}\label{def:heeg}
\textrm{$N^-$ is the square-free product of an \emph{even} number of primes.}\tag{{Heeg}}
\end{equation}

Let $K_\infty/K$ be the $\bZ_p^2$-extension of $K$, and set $\Gamma_K={\rm Gal}(K_\infty/K)$.
We may decompose
\[
H_{p^\infty}\simeq\Delta\times\Gamma_K
\]
for a finite group $\Delta$.
The Galois group ${\rm Gal}(K/\bQ)$ acts naturally on $\Gamma_K$;
let $\Gamma^\cyc\subseteq\Gamma_K$ be the fixed part by this action, and set $\Gamma^\ac:=\Gamma_K/\Gamma^\cyc$.
Then $\Gamma^\ac\simeq{\rm Gal}(K^\ac_\infty/K)$ is identified with the Galois group of the \emph{anticyclotomic}
$\bZ_p$-extension of $K$, on which we have $\tau\sigma\tau^{-1}=\sigma^{-1}$ for the non-trivial element $\tau\in{\rm Gal}(K/\bQ)$.
Similarly, we say that a character $\psi$ is \emph{anticyclotomic} if
$\psi^\rho=\psi^{-1}$.

Assume that $a_p(f)\in\bZ_p^\times$, so that the $p$-adic $L$-function $L_p(f/K,\Sigma^{(1)})$
introduced in Theorem~\ref{thm:hida} is defined, and let $L_p^{\ac}(f/K,\Sigma^{(1)})$ denote its image under the
map 
$\bQ_p\otimes_{\bZ_p}\ro[[H_{p^\infty}]]\rightarrow\bQ_p\otimes_{\bZ_p}\ro[[\Gamma^\ac]]$
induced by the natural projection $H_{p^\infty}\twoheadrightarrow\Gamma^\ac$.

\begin{prop}\label{thm:hida-1}
If the generalized Heegner hypothesis ${\rm (Heeg)}$ holds, then
$L_p^\ac(f/K,\Sigma^{(1)})$ is identically zero.
\end{prop}

\begin{proof}
If $\psi$ is any finite order character of $K$ as in the range of interpolation for $L_p(f/K,\Sigma^{(1)})$
which is anticyclotomic, then the Rankin--Selberg $L$-function $L(f/K,\psi,s)$ is self-dual, with
a functional equation relating its values at $s$ and $2-s$. By hypothesis (Heeg), the sign in this
functional equation is $-1$ (see e.g. \cite[\S{3}]{zhang-AJM}), and so $L(f/K,\psi,1)=0$. The result
thus follows from the interpolation property of $L_p(f/K,\Sigma^{(1)})$.
\end{proof}

In contrast with the immediate proof of Proposition~\ref{thm:hida-1}, the understanding of
$L_p^\ac(f/K,\Sigma^{(2)})$, defined as the image of the $p$-adic $L$-function $L_p(f/K,\Sigma^{(2)})$
of Theorem~\ref{thm:hida} under the natural projection
${\rm Frac}(\ro[[H_{p^\infty}]])\rightarrow{\rm Frac}(\ro[[\Gamma^\ac]])$, requires more work.

To proceed, we need to recall the properties of another anticyclotomic $p$-adic $L$-function.
Let $K_\Delta$ be the fixed field of $\Gamma_K$ in $K(p^\infty)$, so that
${\rm Gal}(K_\Delta/K)\simeq\Delta$. The compositum $K_\infty^\ac K_\Delta$ contains $K[p^\infty]=\cup_{n\geqslant 0}K[p^n]$,
where $K[m]$ denotes the ring class field of $K$ of conductor $m$.
Let $\bQ_p^{\rm nr}$ be the maximal unramified extension of $\bQ_p$,
and denote by $\unr$ the completion of its ring of integers.
Also, let
\[
\rho_f:
{\rm Gal}(\overline{\bQ}/\bQ)\longrightarrow{\rm GL}_2(\bQ_p)
\]
be the Galois representation associated with $f$, and denote by $\bar{\rho}_f$ 
the corresponding semisimple residual representation.

\begin{thm}\label{thm:bdp}
Assume that $p=\pp\overline{\pp}$ splits in $K$ and that hypothesis $({\rm Heeg})$ holds.
There exists a $p$-adic $L$-function
$\mathscr{L}^{\tt BDP}_{\pp}(f/K)\in\unr[[\Gamma^\ac]]$ such that if
$\hat{\psi}:\Gamma^\ac\rightarrow\bC_p^\times$ is the $p$-adic avatar of an unramified Hecke character $\psi$
with infinity type $(-\ell,\ell)$ with $\ell\geqslant 1$, then
\begin{align*}
\biggl(\frac{\mathscr{L}_{\pp}^{\tt BDP}(f/K)(\hat\psi)}{\Omega_p^{2\ell}}\biggr)^2&=
\Gamma(\ell)\Gamma(\ell+1)
\cdot(1-p^{-1}\psi(\pp)\alpha)^2\cdot(1-p^{-1}\psi(\pp)\beta)^2
\cdot \frac{L(f/K,\psi,1)}{\pi^{2\ell+1}\cdot\Omega_K^{4\ell}},
\end{align*}
where 
$\Omega_p\in\unr^\times$ and $\Omega_K\in\bC^\times$  are CM periods. 
Moreover, if $\bar{\rho}_f$ is absolutely irreducible, then
$\mathscr{L}_{\pp}^{\tt BDP}(f/K)$ does not vanish identically.
\end{thm}

\begin{proof}
See \cite[\S{3.3}]{cas-hsieh1} for the construction of
$\mathscr{L}_{\pp}^{\tt BDP}(f/K)$ (see also \cite[\S{5.2}]{burungale-II} for the case $N^-\neq 1$),
and \cite[Thm.~3.9]{cas-hsieh1} for the proof of its nontriviality.
\end{proof}



Note that the projection $L_p^\ac(f/K,\Sigma^{(2)})$ and the square of the $p$-adic $L$-function
$\mathscr{L}^{\tt BDP}_{\pp}(f/K)$ are defined by the interpolation of the same $L$-values.
However, the archimedean periods used in their normalization are different, and therefore
these $p$-adic $L$-functions need not be equal.
In fact, as we shall see below, 
the ratio between these different periods is interpolated by an anticyclotomic projection of a
Katz $p$-adic $L$-function.

Recall that the Hecke $L$-function of $\psi\in\Xi_K$ is defined by
(the analytic continuation of) the Euler product
\[
L_{}(\psi,s)=\prod_{\frakl}\biggl(1-\frac{\psi(\frakl)}{N(\frakl)^s}\biggr)^{-1},
\]
where 
$\frakl$ runs over all prime ideals of $K$, with the convention
that $\psi(\frakl)=0$ for $\frakl\mid\frakf_\psi$. The set of infinity types
of $\psi\in\Xi_K$ for which $s=0$ is a critical value of $L(\psi,s)$ can be written as the disjoint union
$\Sigma_K\sqcup\Sigma_K'$, where $\Sigma_K=\{(\ell_1,\ell_2):0<\ell_1\leqslant\ell_2\}$ and
$\Sigma_K=\{(\ell_1,\ell_2):0<\ell_2\leqslant\ell_1\}$.

\begin{thm}[Katz]\label{thm:katz}
Assume that $p=\pp\overline{\pp}$ splits in $K$. Then there exists a $p$-adic $L$-function
$\mathscr{L}_{\pp}^{}(K)\in\unr[[H_{p^\infty}]]$
such that if $\psi\in\Xi_K$ has trivial conductor 
and infinity type $(\ell_1,\ell_2)\in\Sigma_K$, then
\[
\mathscr{L}^{}_{\pp}(K)(\hat\psi)=\biggl(\frac{\sqrt{D_K}}{2\pi}\biggr)^{\ell_1}
\cdot\Gamma(\ell_2)\cdot(1-\psi(\pp))\cdot(1-p^{-1}\psi^{-1}(\overline\pp))
\cdot\frac{\Omega_p^{\ell_2-\ell_1}}{\Omega_K^{\ell_2-\ell_1}}\cdot L_{}(\psi,0),
\]
where $\Omega_K$ and $\Omega_p$ are as in Theorem~\ref{thm:bdp}. Moreover,
it satisfies the functional equation
\[
\mathscr{L}^{}_\pp(K)(\hat\psi^\rho)=\mathscr{L}^{}_\pp(K)(\hat\psi^{-1}{\rm \mathbf{N}}_K).
\]
\end{thm}

\begin{proof}
See \cite[\S{5.3.0}]{Katz49} or \cite[Thm.~II.4.14]{de_shalit} for the construction,
and \cite[\S{5.3.7}]{Katz49} or \cite[Thm.~II.6.4]{de_shalit} for the functional equation.
\end{proof}

Let $\mathscr{L}_\pp^\ac(K)$ be the image of the Katz $p$-adic $L$-function
$\mathscr{L}^{}_\pp(K)$ under the natural projection $\unr[[H_{p^\infty}]]\rightarrow\unr[[\Gamma^\ac]]$.

\begin{thm}\label{thm:factorization}
Assume that $p=\pp\overline{\pp}$ splits in $K$ and that hypothesis {\rm (Heeg)} hold. Then
\[
L_p^\ac(f/K,\Sigma^{(2)})(\hat\psi)
=\frac{w_K}{h_K}\cdot\frac{\mathscr{L}_{\pp}^{\tt BDP}(f/K)^2(\hat\psi)}{\mathscr{L}_\pp^\ac(K)(\hat\psi^{\rho-1})}
\]
up to a unit in $\ro[[\Gamma^\ac]]^\times$,
where $w_K=\vert\cO_K^\times\vert$ and $h_K$ is the class number of $K$.
In particular, if $\bar{\rho}_f$ is absolutely irreducible, then
$L_p^\ac(f/K,\Sigma^{(2)})$ does not vanish identically.
\end{thm}

\begin{proof}
In the following, for any two $\bC_p$-valued functions $f_1$ and $f_2$ defined on the characters of $\Gamma^\ac$,
we shall write $f_1\sim f_2$ to indicate that their ratio is interpolated by
an invertible Iwasawa function in $\ro[[\Gamma^\ac]]^\times$. We first note that,
when restricted to anticyclotomic characters, the Euler-like factor $\mathcal{E}(\psi,f)$
in Theorem~\ref{thm:hida} becomes a square; in fact, if $\hat\psi:\Gamma^\ac\rightarrow\bC_p^\times$
has trivial conductor and infinity type $(-\ell,\ell)$ with $\ell\geqslant 1$, we see that
\begin{equation}\label{eq:Hida}
\begin{split}
L_p(f/K,\Sigma^{(2)})(\hat\psi)\;\sim&\;
\Gamma(\ell)\Gamma(\ell+1)\cdot\frac{(1-p^{-1}\psi(\pp)\alpha)^2\cdot(1-p^{-1}\psi(\pp)\beta)^2}
{(1-\psi^{1-\rho}(\pp))\cdot(1-p^{-1}\psi^{1-\rho}(\pp))}\\
&\times\frac{L(f/K,\psi,1)}{\pi^{2\ell+1}\cdot
\langle\theta_{\psi_{\ell}},\theta_{\psi_{\ell}}\rangle_M}.
\end{split}
\end{equation}

We will have use for the following result.

\begin{lem}\label{lem:3.7}
We have
\[
\pi^{2\ell+1}\cdot\langle\theta_{\psi_\ell},\theta_{\psi_\ell}\rangle_M
\;\sim\;\frac{h_K}{w_K}\cdot\Gamma(2\ell+1)\cdot L(\psi_\ell^{1-\rho},2\ell+1).
\]
\end{lem}

\begin{proof}
Denote by $\varepsilon_K$ the quadratic character associated with $K$.
Since $\psi_\ell=\psi{\rm \mathbf{N}}_K^\ell$ has infinity type $(-2\ell,0)$,
the result follows immediately after setting $s=2\ell+1$ in the factorization
\[
L({\rm Ad}^0(\theta_{\psi_\ell}),s)=L(\varepsilon_K,s-2\ell)\cdot L(\psi_\ell^{1-\rho},s),
\]
and using \cite[Thm.~5.1]{hida81} and Dirichlet's class number formula (see \cite[Lem.~3.7]{DLR}).
\end{proof}

Continuing with the proof of Theorem~\ref{thm:factorization},
we see that $L(\psi_\ell^{1-\rho},2\ell+1)=L_{}(\psi^{1-\rho},1)$
is interpolated by the value $\mathscr{L}_\pp(K)(\hat{\psi}^{1-\rho}{\rm\mathbf{N}}_K)$
of the Katz $p$-adic $L$-function. Using the functional equation of $\mathscr{L}_{\pp}(K)$, and
combining $(\ref{eq:Hida})$ and Lemma~\ref{lem:3.7} with the interpolation property
in Theorem~\ref{thm:katz}, we thus obtain
\begin{align*}
L_{p}(f/K,\Sigma^{(2)})(\hat\psi)\cdot\mathscr{L}_{\pp}(K)(\hat\psi^{\rho-1})\cdot\frac{h_K}{w_K}\;
\sim&\;\frac{\Gamma(\ell)\Gamma(\ell+1)}{\pi^{2\ell+1}}\cdot(1-p^{-1}\psi(\pp)\alpha)^2\cdot(1-p^{-1}\psi(\pp)\beta)^2\\
&\times\frac{\Omega_p^{4\ell}}{\Omega_K^{4\ell}}\cdot L(f/K,\psi,1),
\end{align*}
which compared with the interpolation property 
in Theorem~\ref{thm:bdp} yields the result. 
\end{proof}

\subsection{Hida's $p$-adic Rankin $L$-series, II}\label{subsec:2varL-II}

Recall from $\S\ref{sec:anti-L}$ the decomposition $H_{p^\infty}\simeq\Delta\times\Gamma_K$, set
$\Lambda=\ro[[\Gamma_K]]$ and $\Lambda_{\unr}=\unr[[\Gamma_K]]$,
and continue to denote by
\[
L_{p}(f/K,\Sigma^{(2)})\in{\rm Frac}(\Lambda)\quad\quad\textrm{and}\quad\quad
\mathscr{L}_{\pp}(K)\in\Lambda_{\unr}
\]
the natural projections of the 
$p$-adic $L$-functions $L_p(f/K,\Sigma^{(2)})$ and $\mathscr{L}_\pp(K)$ of
Theorem~\ref{thm:hida} and Theorem~\ref{thm:katz}, respectively.

\begin{thm}\label{prop:wan}
Assume that $p=\pp\overline{\pp}$ splits in $K$. There exists a $p$-adic $L$-function
$\mathscr{L}_{\pp}(f/K)\in{\rm Frac}(\Lambda_{\unr})$
such that if $\hat{\psi}:\Gamma\rightarrow\bC_p^\times$ has trivial conductor and
infinity type $(\ell_1,\ell_2)\in\Sigma^{(2)}$, then
\begin{align*}
\mathscr{L}_{\pp}(f/K)(\hat\psi)=
&\frac{\Gamma(\ell_2)\Gamma(\ell_2+1)}{\pi^{2\ell_2+1}}
\cdot\mathcal{E}(\psi,f)
\cdot\frac{\Omega_p^{2(\ell_2-\ell_1)}}{\Omega_K^{2(\ell_2-\ell_1)}}\cdot L(f/K,\psi,1),
\end{align*}
where
\[
\mathcal{E}(\psi,f)=(1-p^{-1}\psi(\pp)\alpha)(1-p^{-1}\psi(\pp)\beta)
(1-\psi^{-1}(\overline\pp)\alpha^{-1})(1-\psi^{-1}(\overline\pp)\beta^{-1}),
\]
and $\Omega_K$ and $\Omega_p$ are as in Theorem~\ref{thm:bdp}. Moreover,
it differs from the product
\[
L_{p}(f/K,\Sigma^{(2)})(\hat\psi)
\cdot\frac{h_K}{w_K}\cdot\mathscr{L}_\pp^\ac(K)(\hat\psi^{\rho-1})
\]
by a unit in $\Lambda_{\unr}^\times$.
\end{thm}

\begin{proof}
This follows from exactly the same calculation as
in the proof of Theorem~\ref{thm:factorization}.
\end{proof}

\begin{cor}\label{cor:wan-bdp}
Assume that $p=\pp\overline{\pp}$ splits in $K$ and 
hypothesis $({\rm Heeg})$ holds, and denote by $\mathscr{L}_\pp^\ac(f/K)$
the image of the $p$-adic $L$-function $\mathscr{L}_\pp(f/K)$
of Theorem~\ref{prop:wan} under the natural projection
$\Lambda_{\unr}\rightarrow\Lambda_{\unr}^\ac$. Then
\[
\mathscr{L}_\pp^\ac(f/K)=\mathscr{L}^{\tt BDP}_\pp(f/K)^2
\]
up to a unit in $(\Lambda_{\unr}^\ac)^\times$.
\end{cor}

\begin{proof}
This is clear from Theorem~\ref{thm:factorization}
and the last claim in Theorem~\ref{prop:wan}.
\end{proof}

\section{Iwasawa theory}\label{sec:Iw}

Throughout this section, we let $f\in S_{2}(\Gamma_0(N_f))$, $K/\bQ$, and
$p\geqslant 5$ be as in $\S\ref{sec:padicL}$. In particular, $p=\pp\overline{\pp}$
splits in $K$, and we assume in addition that $p$ is ordinary for $f$, i.e., $a_p(f)\in\bZ_p^\times$.
\sk

Let $V$ be a $\bQ_p$-vector space affording the Galois representation
$\rho_f:{\rm Gal}(\overline{\bQ}/\bQ)\rightarrow{\rm GL}_2(\bQ_p)$ attached to $f$,
fix a Galois-stable $\bZ_p$-lattice $T\subseteq V$, and set $A:=V/T$. By $p$-ordinarity,
there exists a ${\rm Gal}(\overline{\bQ}_p/\bQ_p)$-stable filtration
\begin{equation}\label{eq:Gr-f}
0\longrightarrow\mathscr{F}^+V\longrightarrow V\longrightarrow\mathscr{F}^-V\longrightarrow 0\nonumber
\end{equation}
with both $\mathscr{F}^+V$ and $\mathscr{F}^-V:=V/\mathscr{F}^+V$ one-dimensional over $\bQ_p$
and with the Galois action on $\mathscr{F}^-V$ being unramified. Set $\mathscr{F}^+T:=T\cap\mathscr{F}^+T$,
$\mathscr{F}^-T:=T/\mathscr{F}^+T$, $\mathscr{F}^+A:=\mathscr{F}^+V/\mathscr{F}^+T$, and
$\mathscr{F}^-A:=A/\mathscr{F}^+A$.

\subsection{Selmer groups}\label{sec:selmer}

Let $\Sigma$ be a finite set of places of $K$
containing the places where $V$ is ramified and the places dividing $p\infty$,
and for any finite extension $F$ of $K$, let $\mathfrak{G}_{F,\Sigma}$ denote
the Galois group of the maximal extension of $F$ unramified outside the places above $\Sigma$.

In the next two definitions, we let $M$ denote either $V$, $T$, or $A$.

\begin{defn}\label{def:Sel}
The \emph{Greenberg Selmer group} of $M$ over $F$ is
\begin{equation}
{\rm Sel}(F,M)={\rm ker}\Biggl\{H^1(\mathfrak{G}_{F,\Sigma},M)
\longrightarrow\prod_{v\in\Sigma}\frac{H^1(F_v,M)}{H^1_{\rm Gr}(F_v,M)}\Biggr\},\nonumber
\end{equation}
where
\[
H^1_{\rm Gr}(F_v,M)=
\left\{
\begin{array}{ll}
{\rm ker}\{H^1(F_v,M)\longrightarrow H^1(F_v^{\rm nr},\mathscr{F}^-M)\}&
\textrm{if $v\mid p$};
\\
{\rm ker}\{H^1(F_v,M)\longrightarrow H^1(F_v^{\rm nr},M\}&\textrm{else}.
\\
\end{array}
\right.
\]
\end{defn}

We will also have use for certain modified Selmer groups cut out by different
local conditions at the places above $p$.

\begin{defn}
For $v\mid p$ and $\mathscr{L}_v\in\{\emptyset,{\rm Gr},0\}$, set
\[
H^1_{\mathscr{L}_v}(F_v,M):=
\left\{
\begin{array}{ll}
H^1(F_v,M)&\textrm{if $\mathscr{L}_v=\emptyset$;}\\
H_{\rm Gr}^1(F_v,M)&\textrm{if $\mathscr{L}_v={\rm Gr}$;}\\
\{0\}&\textrm{if $\mathscr{L}_v=0$,}
\end{array}
\right.
\]
and for $\mathscr{L}=\{\mathscr{L}_v\}_{v\vert p}$, define
\begin{equation}\label{def:auxsel}
{\rm Sel}_{\mathscr{L}}(F,M):={\rm ker}\Biggr\{H^1(\mathfrak{G}_{F,\Sigma},M)
\longrightarrow\prod_{\substack{v\in\Sigma\\ v\nmid p}}\frac{H^1(F_v,M)}{H^1_{\rm Gr}(F_v,M)}\times\prod_{v\vert p}
\frac{H^1(F_{v},M)}{H_{\mathscr{L}_v}^1(F_{v},M)}\Biggr\}.\nonumber
\end{equation}
\end{defn}

Of course, when $\mathscr{L}$ is given by
$\mathscr{L}_v={\rm Gr}$ for all $v\mid p$ we recover the previous definition,
in which case $\mathscr{L}$ will be omitted from the notation.

\subsubsection*{Two-variable Selmer groups}

Recall that $\Gamma_K={\rm Gal}(K_\infty/K)$ denotes the Galois group of
the $\bZ_p^2$-extension of $K$, let $\Lambda=\ro[[\Gamma_K]]$ be the associated Iwasawa algebra,
and define the $\Lambda$-modules
\[
\mathbf{T}:=T\otimes_{\ro}\Lambda,
\quad\quad\mathbf{A}:=T\otimes_{\ro}\Lambda^*,
\]
where $\Lambda^*={\rm Hom}_{\bZ_p}(\Lambda,\bQ_p/\bZ_p)$ is the Pontrjagin dual of $\Lambda$.
We equip these modules with the $G_K$-action given by $\rho_f\otimes\can^{-1}$, where
$\can:G_K\twoheadrightarrow\Gamma_K\hookrightarrow\Lambda^\times$ is
the natural character.

Setting
\[
\mathscr{F}^\pm\mathbf{T}:=\mathscr{F}^\pm T\otimes_{\bZ_p}\Lambda,
\quad\quad\mathscr{F}^\pm\mathbf{A}:=\mathscr{F}^\pm T\otimes_{\bZ_p}\Lambda^*,
\]
the Selmer groups ${\rm Sel}_{\mathscr{L}}(K,\mathbf{T})$
and ${\rm Sel}_{\mathscr{L}}(K,\mathbf{A})$ are defined as before,
and by Shapiro's lemma we then have canonical $\Lambda$-module isomorphisms
\begin{equation}\label{eq:Shapiro}
{\rm Sel}_{\mathscr{L}}(K,\mathbf{T})\simeq\varprojlim_{K\subseteq F\subseteq K_\infty}
{\rm Sel}_{\mathscr{L}}(F,T),\quad\quad
{\rm Sel}_{\mathscr{L}}(K,\mathbf{A})\simeq\varinjlim_{K\subseteq F\subseteq K_\infty}
{\rm Sel}_{\mathscr{L}}(F,A),
\end{equation}
where the limits are with respect to the corestriction and restriction map,
respectively, as $F$ runs over the finite extensions of $K$ contained in $K_\infty$
(see e.g. \cite[Prop.~3.4]{SU}). Finally, set 
\[
X_{\mathscr{L}}(K,\mathbf{A}):={\rm Hom}_{\bZ_p}({\rm Sel}_{\mathscr{L}}(K,\mathbf{A}),\bQ_p/\bZ_p),
\]
omitting $\mathscr{L}$ from the notation in the same case as before.

\subsubsection*{Anticyclotomic Selmer groups}

For the ease of notation, set
\[
\Lambda^\ac=\ro[[\Gamma^\ac]].
\]
The Selmer groups ${\rm Sel}_{\mathscr{L}}(K,\Tc)$ and $X_{\mathscr{L}}(K,\Ac)$
are defined by replacing $\Lambda$ with $\Lambda^\ac$ in the above definitions.
For any $\Lambda^\ac$-module $M$, let $M^\iota$ be the underlying abelian group $M$ with
the original $\Lambda^\ac$-module structure composed with the involution
$\iota:\Lambda^\ac\rightarrow\Lambda^\ac$ given by inversion on group-like elements, and
let $M_{\rm tors}$ denote the $\Lambda^\ac$-torsion submodule of $M$. Also, let
$M_{\rm div}$ be the maximal divisible submodule of $M$.

\begin{lem}\label{lem:str-rel}
\hfill
\begin{enumerate}
\item{} ${\rm rank}_{\Lambda^\ac}({\rm Sel}_{\Gr}(K,\Tc))
={\rm rank}_{\Lambda^\ac}(X_{\Gr}(K,\Ac))$.
\item{} ${\rm rank}_{\Lambda^\ac}(X_{{\rm Gr},\emptyset}(K,\Ac))
=1+{\rm rank}_{\Lambda^\ac}(X_{{\rm Gr},0}(K,\Ac))$.
\item{} $Ch_{\Lambda^\ac}(X_{{\rm Gr},\emptyset}(K,\Ac)_{\rm tors})
=Ch_{\Lambda^\ac}(X_{{\rm Gr},0}(K,\Ac)_{\rm tors})$.
\end{enumerate}
\end{lem}

\begin{proof}
The first statement follows easily from \cite[Prop.~2.2.8]{howard-PhD-I} (see \cite[Lem.~3.5]{wan}),
so we just need to prove the other two, for which we will adapt the arguments in \cite[\S{1.2}]{AHsplit}.

For any finite order character $\psi:\Gamma^\ac\rightarrow\mathfrak{O}_L^\times$
with values in the ring of integers of a finite extension $L/\bQ_p$, we continue to denote
by $\psi$ its natural linear extension $\Lambda^\ac\rightarrow\mathfrak{O}_L$,
and set $A(\psi):=\Ac\otimes_{\Lambda^\ac,\psi}\mathfrak{O}_L$ equipped with the diagonal
$G_K$-action. Then, by \cite[Lem.~3.5.3, Thm.~4.1.13]{MR-KS}
there is a non-canonical isomorphism
\begin{equation}\label{eq:MR}
H^1_{{\rm Gr},\emptyset}(K,A(\psi))[p^i]\simeq(L/\mathfrak{O}_L)^r[p^i]
\oplus H^1_{{\rm Gr},0}(K,A(\psi^{-1}))[p^i],
\end{equation}
where
\begin{itemize}
\item{} $H^1_{{\rm Gr},\emptyset}(K,A(\psi))$ is the subgroup ${\rm Sel}_{{\rm Gr},\emptyset}(K,A(\psi))$
consisting of classes whose restriction at $\pp$ (resp. $\overline{\pp}$)
lies in $H^1(K_\pp,A(\psi))_{\rm div}$ (resp. $H^1_{\rm Gr}(K_{\overline\pp},A(\psi))_{\rm div}$);
\item{} $H^1_{{\rm Gr},0}(K,A(\psi^{-1}))$ is the subgroup of ${\rm Sel}_{{\rm Gr},0}(K,A(\psi^{-1}))$
consisting of classes whose restriction to $\pp$ lies in $H^1_{\rm Gr}(K_\pp,A(\psi))_{\rm div}$;
\item{} and $r\geqslant 0$ is the \emph{core rank} (see \cite[Def.~4.1.11]{MR-KS})
of the Selmer conditions defining $H^1_{{\rm Gr},\emptyset}(K,A(\psi))$.
\end{itemize}
By \cite[Thm.~2.18]{DDT}, the value of $r$ is given by
\begin{equation}\label{eq:DDT}
{\rm corank}_{\ro}H^1(K_\pp,\mathscr{F}^+A(\psi))
+{\rm corank}_{\ro}H^1(K_{\overline{\pp}},A(\psi))
-{\rm corank}_{\ro}H^0(K_{v},A(\psi)),
\end{equation}
where $v$ denotes the unique archimedean place of $K$. By the local Euler characteristic formula,
the first two terms in $(\ref{eq:DDT})$ are equal to $1$ and $2$, respectively, while the third one clearly equals $2$.
Thus $r=1$ in $(\ref{eq:MR})$ and letting $i\to\infty$ we conclude that
\begin{equation}\label{eq:MR2}
H^1_{{\rm Gr},\emptyset}(K,A(\psi))\simeq(L/\mathfrak{O}_L)
\oplus H^1_{{\rm Gr},0}(K,A(\psi^{-1})).
\end{equation}

Now, it is easy to show that the natural restriction maps
\begin{align*}
H^1_{{\rm Gr},\emptyset}(K,A(\psi))
&\longrightarrow{\rm Sel}_{{\rm Gr},\emptyset}(K,\Ac)(\psi)^{\Gamma^\ac}\\
H^1_{{\rm Gr},0}(K,A(\psi^{-1}))
&\longrightarrow{\rm Sel}_{{\rm Gr},0}(K,\Ac)(\psi^{-1})^{\Gamma^\ac}
\end{align*}
are injective with finite bounded cokernel as $\psi$ varies (cf. \cite[Lem.~1.2.4]{AHsplit}),
and since
\[
{\rm Sel}_{{\rm Gr},0}(K,\Ac)(\psi^{-1})^{\Gamma^\ac}
\simeq{\rm Sel}_{{\rm Gr},0}(K,\Ac)(\psi)^{\Gamma^\ac}
\]
by the action of complex conjugation, we see that statements $(2)$ and $(3)$ follow from $(\ref{eq:MR2})$
by the same argument as in \cite[Lem.~1.2.6]{AHsplit}. (For $(3)$, note that in the proof of
\emph{loc.cit.} the prime $p\Lambda^{\rm ac}$ is excluded, but this can be dealt with
similarly as in \cite[Thm.~2.2.10]{howard-PhD-I}.)
\end{proof}

\subsection{Beilinson--Flach classes and explicit reciprocity laws}

In this subsection, we briefly recall
the special type of Beilinson--Flach classes from \cite{KLZ2} that we will
need in this paper, and the ``explicit reciprocity laws'' relating them
to $p$-adic $L$-functions.

For any Hida family $\mathbf{f}$, we let $M(\mathbf{f})^*$ be the
associated Galois representation as in \cite[Def.~7.2.5]{KLZ2};
in particular, $M(\mathbf{f})^*$ is a finite and projective module over a local
$\bZ_p[[\bZ_p^\times]]$-algebra $\Lambda_{\mathbf{f}}$, and there is a short
exact sequence of $\Lambda_{\mathbf{f}}[{\rm Gal}(\overline{\bQ}_p/\bQ_p)]$-modules
\begin{equation}\label{eq:ord}
0\longrightarrow\mathscr{F}^+M(\mathbf{f})^*\longrightarrow
M(\mathbf{f})^*\longrightarrow
\mathscr{F}^-M(\mathbf{f})^*\longrightarrow 0\nonumber
\end{equation}
with the Galois action on $\mathscr{F}^-M(\mathbf{f})^*$ being unramified.
If $\mathbf{f}$ is the Hida family associated with a $p$-ordinary newform $f$ as above,
then there is a height one prime
$\mathfrak{P}\subseteq\Lambda_{\mathbf{f}}$ with residue field $\bQ_p$
for which we have an isomorphism
\begin{equation}\label{eq:sp}
T\simeq M(\mathbf{f})^*\otimes_{\Lambda_{\mathbf{f}}}\Lambda_{\mathbf{f}}/\mathfrak{P}\nonumber
\end{equation}
compatible with the filtrations $\mathscr{F}^\pm$. Set $\Lambda_{\rm cyc}=\bZ_p[[{\rm Gal}(K(\mu_{p^\infty})/K)]]$.


\begin{thm}[Kings--Loeffler--Zerbes]
\label{thm:Col}
There exists an element $\mathcal{BF}\in{\rm Sel}_{{\rm Gr},\emptyset}(K,\mathbf{T})$,
a fractional ideal $I_{\mathbf{g}}\subseteq{\rm Frac}(\bZ_p[[H_{\pp^\infty}]])$,
and $\Lambda$-linear injections with pseudo-null cokernel:
\begin{align*}
{\rm Col}^{(1)}:H^1(K_{\overline{\pp}},\mathscr{F}^{-}\mathbf{T})
\;&\longrightarrow\;\Lambda\otimes_{\ro}\bQ_p,\\
{\rm Col}^{(2)}:H^1(K_{\pp},\mathscr{F}^{+}\mathbf{T})
\;&\longrightarrow\;I_{\mathbf{g}}\hat{\otimes}_{\bZ_p}\Lambda_{\rm cyc}\otimes_{\bZ_p}\bQ_p,
\end{align*}
such that
\begin{align*}
{\rm Col}^{(1)}({\rm loc}_{\overline\pp}(\mathcal{BF}))&=L_p(f/K,\Sigma^{(1)}),\quad\quad
{\rm Col}^{(2)}({\rm loc}_\pp(\mathcal{BF}))=L_p(f/K,\Sigma^{(2)}),
\end{align*}
where $L_p(f/K,\Sigma^{(i)})$ are the $p$-adic $L$-functions of Theorem~\ref{thm:hida}.
\end{thm}

\begin{proof}
The fields $K(\frakp^\infty)$ and $K(\mu_{p^\infty})$
are linearly disjoint over $K$ and their compositum is $K(p^\infty)$, and so
\begin{equation}\label{eq:disj}
\bZ_p[[H_{p^\infty}]]
\simeq\bZ_p[[H_{\frakp^\infty}]]
\hat\otimes_{\bZ_p}\Lambda_{\rm cyc}
\end{equation}
as Galois modules, where $\bZ_p[[H_{p^\infty}]]$ is equipped with the $G_K$-action
induced by the inverse of the character $G_K\twoheadrightarrow H_{p^\infty}\hookrightarrow\bZ_p[[H_{p^\infty}]]^\times$,
and similarly for the terms on the right-hand side.
Let $\mathbf{f}$ be the Hida family attached to $f$,
and consider the formal $q$-expansion
\begin{equation}\label{eq:CM}
\mathbf{g}=\sum_{(\fraka,\pp)=1}[\fraka]q^{N(\fraka)}\in\Lambda_\pp[[q]],
\end{equation}
where $\Lambda_\pp=\bZ_p[[H_{\frakp^\infty}]]$.
Taking $m=1$, $c>1$ an integer coprime to $6N_fD_Kp$, and $N$ a positive integer divisible
by $N_f$ and $D_K$ and with the same prime factors as $N_fD_K$, the Beilinson--Flach class
$_{c}\mathcal{BF}_m^{\mathbf{f},\mathbf{g}}$ constructed in \cite[Def.~8.1.1]{KLZ2} gives rise to an element
\begin{equation}\label{eq:def-BF}
_{c}\mathcal{BF}^{\mathbf{f},\mathbf{g}}\in H^1(\mathfrak{G}_{\bQ,S},
M(\mathbf{f})^*\hat\otimes_{\bZ_p} M(\mathbf{g})^*\hat\otimes_{\bZ_p}\Lambda_\Gamma),
\end{equation}
where $\mathfrak{G}_{\bQ,S}$ is the Galois group of the maximal extension of $\bQ$ unramified outside the
primes dividing $N_fD_Kp$, $M(\mathbf{f})^*$ and $M(\mathbf{g})^*$ are the Galois
modules associated to the corresponding Hida families as in [\emph{loc.cit.}, Def.~7.2.5],
and $\Lambda_\Gamma:=\bZ_p[[{\rm Gal}(\bQ(\mu_{p^\infty})/\bQ)]]$ is equipped with the inverse of the natural $G_{\bQ}$-action. With
a slight abuse of notation, we continue to denote by $_{c}\mathcal{BF}^{\mathbf{f},\mathbf{g}}$ the image
of the class $(\ref{eq:def-BF})$ in $H^1(\bQ,M(\mathbf{f})^*\hat\otimes_{} M(\mathbf{g})^*\hat\otimes_{}\Lambda_\Gamma)$
under inflation.

The CM Hida family $\mathbf{g}$ satisfies $M_{}(\mathbf{g})^*\simeq{\rm Ind}_K^\bQ\Lambda_\pp$, where $\Lambda_\pp$
is equipped with the $G_K$-action described in $(\ref{eq:disj})$ above (see \cite[Thm.~5.2.4]{LLZ-K}),
and we let $_{c}\mathcal{BF}$ denote the image of $_{c}\mathcal{BF}^{\mathbf{f},\mathbf{g}}$ under the composite map
\begin{align*}
H^1(\bQ,M(\mathbf{f})^*\hat{\otimes}_{\bZ_p} M_{}(\mathbf{g})^*\hat\otimes_{\bZ_p}\Lambda_{\Gamma})
&\longrightarrow H^1(\bQ,T\hat{\otimes}_{\bZ_p} M_{}(\mathbf{g})^*\hat\otimes_{\bZ_p}\Lambda_{\Gamma})\\
&\overset{\simeq}\longrightarrow H^1(K,T\hat\otimes_{\bZ_p}\bZ_p[[H_{p^\infty}]])\\
&\longrightarrow H^1(K,\mathbf{T}),
\end{align*}
where the first arrow is the natural map induced by the specialization $M(\mathbf{f})^*\rightarrow T$,
the second isomorphism is given by Shapiro's lemma, and the third arrow is induced by the projection
$H_{p^\infty}\twoheadrightarrow\Gamma_K$. By \cite[Lem.~6.8.9]{LLZ},
we may dispense with $c$ by noting that the tame character of
$\mathbf{g}$ is non-trivial at some prime dividing $D_K$; thus we have defined a class
$\mathcal{BF}\in H^1(K,\mathbf{T})$.

By its geometric construction, the resulting class $\mathcal{BF}$ is unramified
outside $p$, and the vanishing of its natural image in $H^1(K_{\pp},\mathscr{F}^-\mathbf{T})$
follows from \cite[Prop.~8.1.7]{KLZ2}; therefore, we have $\mathcal{BF}\in{\rm Sel}_{{\rm Gr},\emptyset}(K,\mathbf{T})$.

Now let $\omega_{\mathbf{f}}$, $\eta_{\mathbf{f}}$, $\omega_{\mathbf{g}}$ and $\eta_{\mathbf{g}}$
be as constructed in \cite[Prop.~10.1.1]{KLZ2}, and define ${\rm Col}^{(1)}$ and  ${\rm Col}^{(2)}$
to be the specializations at $f$ of the maps
\begin{align*}
\left\langle\mathcal{L}(-),\eta_{\mathbf{f}}\otimes\omega_{\mathbf{g}}\right\rangle:
H^1(K_{\overline\pp},\mathscr{F}^{-+}M(\mathbf{f}\otimes\mathbf{g})^*\hat\otimes_{\bZ_p}\Lambda_\Gamma)
\;&\longrightarrow\;
\left(I_\mathbf{f}\hat\otimes_{\bZ_p}\Lambda^{}_{\pp}\hat\otimes_{\bZ_p}\Lambda_{\Gamma}\right)\otimes_{\bZ_p}\bZ_p[\mu_N],\\
\left\langle\mathcal{L}(-),\eta_{\mathbf{g}}\otimes\omega_{\mathbf{f}}\right\rangle:
H^1(K_{\pp},\mathscr{F}^{+-}M(\mathbf{f}\otimes\mathbf{g})^*\hat\otimes_{\bZ_p}\Lambda_\Gamma)
\;&\longrightarrow\;
\left(I_{\mathbf{g}}\hat\otimes_{}\Lambda_{\mathbf{f}}\hat\otimes_{\bZ_p}\Lambda_{\Gamma}\right)\otimes_{\bZ_p}\bZ_p[\mu_N],
\end{align*}
obtained by composing the ``big logarithm map'' $\mathcal{L}$ of \cite[Thm.~8.2.8]{KLZ2}
with the pairing against the classes $\eta_{\mathbf{f}}\otimes\omega_{\mathbf{g}}$ and $\eta_{\mathbf{g}}\otimes\omega_{\mathbf{f}}$,
respectively. (Here $\mathscr{F}^{-+}M(\mathbf{f}\otimes\mathbf{g})^*$ and $\mathscr{F}^{+-}M(\mathbf{f}\otimes\mathbf{g})^*$ are
the subquotients of $M(\mathbf{f})^*\hat\otimes M(\mathbf{g})^*$ defined in [\emph{loc.cit.}, p.~75],
and $I_{\mathbf{g}}\subseteq{\rm Frac}(\Lambda_\pp)$ and $I_{\mathbf{f}}\subseteq{\rm Frac}(\Lambda_{\mathbf{f}})$
are the \emph{congruence ideal} of the corresponding Hida family; see [\emph{loc.cit.}, Rem.~9.6.3].)

Thus defined, the claim about the kernel and cokernel of the maps ${\rm Col}^{(i)}$ follows
from the last two claims in \cite[Thm.~8.2.3]{KLZ2}, and their claimed relation with the $p$-adic
$L$-functions $L_p(f/K,\Sigma^{(i)})$ is a consequence of the explicit reciprocity law of \cite[Thm.~10.2.2]{KLZ2}.
\end{proof}

Denote by $\mathcal{BF}^\ac$ the image of $\mathcal{BF}$ 
under the natural map $H^1(K,\mathbf{T})\rightarrow H^1(K,\Tc)$.

\begin{cor}\label{cor:str-Gr}
For all primes $v\mid p$, the class ${\rm loc}_v(\mathcal{BF})$ is non-torsion over $\Lambda$. Moreover,
if hypothesis ${\rm(Heeg)}$ holds, then ${\rm loc}_{\pp}(\mathcal{BF}^\ac)$ is non-torsion over $\Lambda^{\ac}$,
and ${\rm loc}_{\overline{\pp}}(\mathcal{BF}^\ac)$ lies in the kernel of the natural map
\[
H^1(K_{\overline\pp},\Tc)\longrightarrow H^1(K_{\overline{\pp}},\mathscr{F}^-\Tc).
\]
In particular, $\mathcal{BF}^\ac\in{\rm Sel}(K,\Tc)$.
\end{cor}

\begin{proof}
In light of 
Theorem~\ref{thm:Col},
the first claim follows from the nonvanishing of the $p$-adic $L$-functions
$L_p(f/K,\Sigma^{(i)})$ 
of Theorem~\ref{thm:hida} (see Remark~\ref{rem:non0});
the second claim from the nonvanishing of the projection $L_p^\ac(f/K,\Sigma^{(2)})$ in Theorem~\ref{thm:factorization};
and the last claim from the vanishing of the projection $L_p^\ac(f/K,\Sigma^{(1)})$ in Proposition~\ref{thm:hida-1}
and the injectivity of ${\rm Col}^{(1)}$.
\end{proof}

\subsection{Two-variable main conjectures}\label{sec:ES}

Recall that the generalized Selmer groups ${\rm Sel}_{\mathscr{L}}(K,\mathbf{T})$ and ${\rm Sel}_{\mathscr{L}}(K,\Tc)$
are submodules of $H^1(\mathfrak{G}_{K,\Sigma},\mathbf{T})$ and $H^1(\mathfrak{G}_{K,\Sigma},\Tc)$, respectively.

\begin{lem}\label{lem:no-tors}
Assume that the Galois representation $G_K\rightarrow{\rm Aut}_{\bZ_p}(T)$ is surjective. Then the modules
$H^1(\mathfrak{G}_{K,\Sigma},\mathbf{T})$ and $H^1(\mathfrak{G}_{K,\Sigma},\Tc)$ are torsion-free
over $\Lambda$ and $\Lambda^\ac$, respectively.
\end{lem}

\begin{proof}
As shown in \cite[Lem.~2.2.9]{howard-PhD-I},
this follows immediately from \cite[\S{1.3.3}]{PR:Lp}
(whose argument readily extends to the two-variable setting).
\end{proof}

%
For any $\ro$-module $M$, define $M_{\unr}:=M\hat{\otimes}_{\ro}{\unr}$. Let
$\gamma^{\ac}\in\Gamma^\ac$ be a topological generator, and let $P\subseteq\Lambda$ be
the pullback of the augmentation ideal of $P_{\ac}:=(\gamma^\ac-1)\subseteq\Lambda^\ac$.

\begin{thm}\label{thm:2-varIMC}
Assume that $G_K\rightarrow{\rm Aut}_{\bZ_p}(T)$ is surjective.
Then the following two statements are equivalent:
\begin{enumerate}
\item{} $X_{{\rm Gr},0}(K,\mathbf{A})$ is $\Lambda$-torsion,
${\rm Sel}_{{\rm Gr},\emptyset}(K,\mathbf{T})$ has $\Lambda$-rank $1$, and
\[
Ch_{\Lambda}(X_{{\rm Gr},0}(K,\mathbf{A}))=
Ch_\Lambda\bigg(\frac{{\rm Sel}_{{\rm Gr},\emptyset}(K,\mathbf{T})}{\Lambda\cdot\mathcal{BF}}\biggr)
\]
in $\Lambda[1/P]$.
\item{} Both $X_{\emptyset,0}(K,\mathbf{A})$ and ${\rm Sel}_{0,\emptyset}(K,\mathbf{T})$ are $\Lambda$-torsion, and
\[
Ch_{\Lambda_{\unr}}(X_{\emptyset,0}(K,\mathbf{A})_{\unr})=(\mathscr{L}_\pp(f/K))
\]
in $\Lambda_{\unr}[1/P]$.
\end{enumerate}
\end{thm}

\begin{proof}
We only prove the implication $(2)\Rightarrow(1)$,
which is the only direction we will use in the following, but the
proof of the converse implication follows from the same ideas.
Assume that ${\rm Sel}_{0,\emptyset}(K,\mathbf{T})$ is $\Lambda$-torsion,
and consider the tautological exact sequence
\[
0\longrightarrow{\rm Sel}_{0,\emptyset}(K,\mathbf{T})\longrightarrow
{\rm Sel}_{{\rm Gr},\emptyset}(K,\mathbf{T})\xrightarrow{{\rm loc}_\pp}
H^1_{\rm Gr}(K_\pp,\mathbf{T}).
\]
By Corollary~\ref{cor:str-Gr} the image of ${\rm loc}_\pp$ is nontorsion over $\Lambda$,
and so (since ${\rm rank}_{\Lambda}(H^1_{\rm Gr}(K_\pp,\mathbf{T}))=1$) the
cokernel of ${\rm loc}_{\pp}$ is $\Lambda$-torsion and ${\rm Sel}_{{\rm Gr},\emptyset}(K,\mathbf{T})$ has
$\Lambda$-rank $1$. Since ${\rm Sel}_{0,\emptyset}(K,\mathbf{T})$ is has no $\Lambda$-torsion by Lemma~\ref{lem:no-tors},
our assumption implies that ${\rm Sel}_{0,\emptyset}(K,\mathbf{T})$ is trivial, and so
Poitou--Tate duality gives rise to the exact sequence
\begin{equation}\label{eq:ES-1b}
0\longrightarrow{\rm Sel}_{{\rm Gr},\emptyset}(K,\mathbf{T})\longrightarrow
H^1_{\rm Gr}(K_{\pp},\mathbf{T})\longrightarrow
X_{\emptyset,0}(K,\mathbf{A})\longrightarrow X_{{\rm Gr},0}(K,\mathbf{A})\longrightarrow 0,\nonumber
\end{equation}
and moding out by the images of $\mathcal{BF}$ 
to the exact sequence
\begin{equation}\label{eq:mod-BF}
0\longrightarrow\frac{{\rm Sel}_{{\rm Gr},\emptyset}(K,\mathbf{T})_{}}{\Lambda_{}\cdot\mathcal{BF}}
\longrightarrow\frac{H^1_{\rm Gr}(K_{\pp},\mathbf{T})_{}}{\Lambda_{}\cdot{\rm loc}_\pp(\mathcal{BF})}
\longrightarrow X_{\emptyset,0}(K,\mathbf{A})_{}
\longrightarrow X_{{\rm Gr},0}(K,\mathbf{A})_{}\longrightarrow 0.
\end{equation}

Now recall the congruence ideal $I_{\mathbf{g}}$ 
of CM forms $\mathbf{g}$ introduced in the proof of Theorem~\ref{thm:Col},
which we shall view as an element in ${\rm Frac}(\Lambda^\ac)$ in the following.
By \cite[Thm.~A.4]{HT-durham} (see also \cite[Thm.~I]{HT-ENS} and \cite[Thm.~1.4.7]{HT-117})
we then have the divisibilities
\begin{equation}\label{eq:HT}
(h_K\cdot\mathscr{L}_\pp^{\ac}(K))\;\supseteq\; I_{\mathbf{g}}^{-1}\supseteq\; (h_K\cdot\mathscr{F}_\pp^{\ac}(K))
\end{equation}
in $\Lambda^\ac_{R_0}[1/P_\ac]$, where $h_K$ is the class number of $K$, $\mathscr{L}_\pp^{\ac}(K)$
is an anticyclotomic projection of the Katz $p$-adic $L$-function of $K$ as in Theorem~\ref{prop:wan},
and $\mathscr{F}_\pp^\ac(K)$ generates the characteristic ideal of the Pontrjagin dual of
the maximal abelian pro-$p$-extension of $K_\infty^\ac$ unramified outside $\pp$.
On the other hand, Rubin's proof \cite{rubin-IMC} of the Iwasawa main conjecture for $K$ yields the equality
\begin{equation}\label{eq:Rubin}
(\mathscr{F}_\pp^\ac(K))=(\mathscr{L}_\pp^\ac(K))
\end{equation}
as ideals in $\Lambda_{R_0}^\ac$. Combining $(\ref{eq:HT})$ and $(\ref{eq:Rubin})$, it follows
that $I_{\mathbf{g}}^{-1}$ is generated by $h_K\cdot\mathscr{L}_\pp^{\rm ac}(K)$ up to powers of $P$,
and hence by Theorem~\ref{thm:Col} and the factorization in Theorem~\ref{prop:wan}, the map
$(h_K\cdot\mathscr{L}_\pp^{\rm ac}(K))\times{\rm Col}^{(2)}$ yields an injection
\begin{equation}\label{eq:katz-col}
\frac{H^1_{\rm Gr}(K_{\pp},\mathbf{T})_{\unr}}{\Lambda_{\unr}\cdot{\rm loc}_\pp(\mathcal{BF})}
\hookrightarrow\frac{\Lambda_{\unr}}{\Lambda_{\unr}\cdot\mathscr{L}_\pp(f/K)}
\end{equation}
after extending scalars to $\Lambda_{\unr}[1/P]$ with pseudo-null cokernel.
By multiplicativity of characteristic ideals in exact sequences,
taking characteristic ideals in $(\ref{eq:mod-BF})$ and $(\ref{eq:katz-col})$, the result follows.
\end{proof}

We record the following result for our later use.

\begin{prop}\label{thm:str}
Assume that equality in part (1) of Theorem~\ref{thm:2-varIMC} holds as ideals in $\Lambda$.
Then
\[
Ch_{\Lambda^\ac}(X_{{\rm Gr},0}(K,\Ac))
=Ch_{\Lambda^\ac}\biggl(\frac{{\rm Sel}_{{\rm Gr},\emptyset}(K,\Tc)}{\Lambda^\ac\cdot\mathcal{BF}^\ac}\biggr)
\]
as ideals in $\Lambda^\ac$.
\end{prop}

\begin{proof}
Of course, this follows from descending 
from $K_\infty$ to $K_\infty^\ac$. 
Let $\gamma^{\rm cyc}\in\Gamma^{\rm cyc}$ be a topological generator, and let $I^{\rm cyc}$ be the principal
ideal $(\gamma^{\rm cyc}-1)\Lambda\subseteq\Lambda$. Then by \cite[Prop.~3.9]{SU} (with
the roles of the cyclotomic and anticyclotomic $\bZ_p$-extensions reversed) we have
\[
X_{{\rm Gr},0}(K,\mathbf{A})/I^{\rm cyc}X_{{\rm Gr},0}(K,\mathbf{A})\simeq
X_{{\rm Gr},0}(K,\Ac).
\]
In particular, $X_{{\rm Gr},0}(K,\Ac)$ is $\Lambda^\ac$-torsion
and by \cite[Lemma~6.2(ii)]{rubin-IMC} it follows that
\begin{equation}\label{eq:3.9}
Ch_{\Lambda^\ac}(X_{{\rm Gr},0}(K,\Ac))
=Ch_{\Lambda}(X_{{\rm Gr},0}(K,\mathbf{A}))\cdot\mathfrak{D},
\end{equation}
where $\mathfrak{D}:=Ch_{\Lambda^\ac}(X_{{\rm Gr},0}(K,\mathbf{A})[I^{\rm cyc}])$.
On the other hand, set
\[
Z(K_\infty):={\rm Sel}_{{\rm Gr},\emptyset}(K,\mathbf{T})/(\mathcal{BF}),
\quad
Z(K_\infty^\ac):={\rm Sel}_{{\rm Gr},\emptyset}(K,\Tc)/(\mathcal{BF}^\ac).
\]
Using the fact that $I^{\rm cyc}$ is principal,
a straightforward application of Snake's Lemma yields the exactness of
\begin{equation}\label{eq:snake}
{\rm Sel}_{{\rm Gr},\emptyset}(K,\mathbf{T})[I^{\rm cyc}]\longrightarrow
Z(K_\infty)[I^{\rm cyc}]\longrightarrow(\mathcal{BF})/I^{\rm cyc}(\mathcal{BF}).
\end{equation}
Arguing as in the proof of \cite[Prop.~2.4.15]{AHsplit} we see that the natural
$\Lambda^\ac$-module map
\[
Z(K_\infty)/I^{\rm cyc}Z(K_\infty)\longrightarrow Z(K_\infty^\ac)
\]
is injective with cokernel having characteristic ideal $\mathfrak{D}$, and hence
\begin{equation}\label{eq:2.4.5}
Ch_{\Lambda^\ac}(Z(K_\infty^\ac))=Ch_{\Lambda^\ac}(Z(K_\infty)/I^{\rm cyc}Z(K_\infty))\cdot\mathfrak{D}.
\end{equation}

Since ${\rm Sel}_{{\rm Gr},\emptyset}(K,\mathbf{T})$ has no $\Lambda$-torsion by Lemma~\ref{lem:no-tors},
the leftmost term in $(\ref{eq:snake})$ vanishes; on the other hand
the rightmost term is clearly torsion-free, and hence $Z(K_\infty)[I^{\rm cyc}]$ is torsion-free.
Since \cite[Lem.~6.2(i)]{rubin-IMC} and equality $(\ref{eq:2.4.5})$ imply
that $Z(K_\infty)[I^{\rm cyc}]$ is also a torsion $\Lambda^\ac$-module
(using the nonvanishing of the terms in that equality), we conclude that $Z(K_\infty)[I^{\rm cyc}]=0$,
and by \cite[Lem.~6.2(ii)]{rubin-IMC} it follows that
\begin{equation}\label{eq:6.2}
Ch_{\Lambda}(Z(K_\infty))\cdot\Lambda^\ac
=Ch_{\Lambda^\ac}(Z(K_\infty)/I^{\rm cyc}Z(K_\infty)).
\end{equation}
Combined with $(\ref{eq:3.9})$, 
we thus arrive at
\begin{equation}
\begin{split}\label{eq:Z}
Ch_{\Lambda^\ac}(X_{{\rm Gr},0}(K,\Ac))
&=Ch_{\Lambda}(X_{{\rm Gr},0}(K,\mathbf{A}))\cdot\mathfrak{D}\\
&=Ch_{\Lambda}(Z(K_\infty))\cdot\mathfrak{D}\\
&=Ch_{\Lambda^\ac}(Z(K_\infty^\ac)),
\end{split}\nonumber
\end{equation}
using $(\ref{eq:2.4.5})$ and $(\ref{eq:6.2})$ for the last equality. This completes the proof.
\end{proof}

\section{$\Lambda$-adic Gross--Zagier formula}

Throughout this section, we let $E/\bQ$ be an elliptic curve of conductor $N$,
$f=\sum_{n=1}^\infty a_n(f)q^n$ 
be the normalized newform of weight $2$ associated with $E$,
and $p\nmid 6N$ be a prime of ordinary reduction for $E$.
We also let $K/\bQ$ be an imaginary quadratic field of discriminant prime to $Np$ satisfying
the generalized Heegner hypothesis (Heeg) in the Introduction and such that $p=\pp\overline{\pp}$ splits in $K$.

\subsection{Heegner point main conjecture}


Recall that for every integer $m\geqslant 1$ we let $K[m]$ be the
ring class field of $K$ of conductor $m$. In particular, $K[1]$ is the Hilbert class field of $K$.

\begin{prop}\label{prop:HP}
There exists a collection of Heegner points $z_{f,p^n}\in E(K[p^n])\otimes_{\bZ}\bZ_p$
satisfying the norm-compatibility relations
\begin{equation}\label{eq:HP-compat}
a_p(f)\cdot z_{f,p^{n-1}}=
\left\{
\begin{array}{ll}
z_{f,p^{n-2}}+{\rm Norm}^{K[p^n]}_{K[p^{n-1}]}(z_{f,p^n})&\textrm{if $n>1$;}\\
\sigma_\pp z_{f,1}+\sigma_{\overline\pp} z_{f,1}+w_K\cdot{\rm Norm}^{K[p]}_{K[1]}(z_{f,p})&\textrm{if $n=1$},
\end{array}
\right.
\end{equation}
where $w_K=\vert\cO_K^\times\vert$, and $\sigma_\pp, \sigma_{\overline\pp}\in{\rm Gal}(K[1]/K)$
are the Frobenius elements of $\pp, \overline{\pp}$, respectively.
\end{prop}

\begin{proof}
This is standard, but for our later reference we briefly recall the construction, referring the reader
to \cite[\S{1.2}]{howard-PhD-II}) and the references therein for further details. Let $X:=X_{N^+,N^-}$
be the Shimura curve attached to a quaternion algebra $B/\bQ$ of discriminant $N^-$ equipped with
the $\Gamma_0(N^+)$-level structure defined be an Eichler order of level $N^+$.
By \cite[Prop.~1.2.1]{howard-PhD-II}, there exists a collection of CM points $h_{p^n}\in X(K[p^n])$
satisfying the norm relations $(\ref{eq:HP-compat})$ with the Hecke correspondence
$T_p$ in place of $a_p(f)$.

By 
a construction due to S.-W. Zhang \cite{zhang-153}
extending a classical result of Shimura, after possibly replacing $E$ by a $\bQ$-isogenous elliptic curve we
may fix a parametrization
\[
\Phi_E:{\rm Jac}(X)\longrightarrow E
\]
defined over $\bQ$. 
If $N^-\neq 1$, the curve $X$ has no cusps, and
an auxiliary choice is necessary in order to embed $X$ into ${\rm Jac}(X)$.
To that end, we fix a prime $\ell\nmid N$ such that $\ell+1-a_\ell(f)$
is a $p$-adic unit,\footnote{In our application, the existence of such a prime $\ell$ will be guaranteed
by a certain big image hypothesis; without this assumption, one may use a certain \emph{Hodge class}
to exhibit an embedding $X\rightarrow{\rm Jac}(X)$.} and define $z_{f,p^n}\in E(K[p^n])\otimes_{\bZ}\bZ_p$
be the image of the degree zero divisor $(\ell+1-a_\ell(f))^{-1}(\ell+1-T_\ell)h_{p^n}$
under the map ${\rm Jac}(X)(K[p^n])\rightarrow E(K[p^n])$ induced by $\Phi_E$.
\end{proof}

Now let $\alpha$ be the unit root of
$x^2-a_p(f)x+p$, and define the \emph{regularized Heegner points}
$z_{f,p^n,\alpha}\in E(K[p^n])\otimes_{\bZ}\bZ_p$ by
\[
z_{f,p^n,\alpha}:=
\left\{
\begin{array}{ll}
z_{f,p^n}-\frac{1}{\alpha}z_{f,p^{n-1}}&\textrm{if $n>0$;}\\
\frac{1}{w_K}\left(1-\frac{\sigma_\pp}{\alpha}\right)\left(1-\frac{\sigma_{\overline{\pp}}}{\alpha}\right)z_{f,1}&\textrm{if $n=0$.}
\end{array}
\right.
\]
By Proposition~\ref{prop:HP}, we then have
\begin{equation}\label{eq:comp}
{\rm Norm}^{K[p^n]}_{K[p^{n-1}]}(z_{f,p^n,\alpha})=\alpha\cdot z_{f,p^{n-1},\alpha}\nonumber
\end{equation}
for all $n\geqslant 1$. Letting $\mathbf{z}_{f,p^n,\alpha}$
be the image of the point $z_{f,p^n,\alpha}$ under the Kummer map
\[
E(K[p^n])\otimes_{\bZ}\bZ_p\longrightarrow
H^1(K[p^n],T_pE),
\]
the classes $\alpha^{-n}\cdot\mathbf{z}_{f,p^n,\alpha}$ are then compatible
under corestriction, thus defining a class
\[
\mathbf{z}_{f}\in H^1(K,\mathbf{T}^\ac)
\]
which 
lands in the Selmer group ${\rm Sel}(K,\mathbf{T}^\ac)$.

\begin{conj}
\label{HP-MC}
Both ${\rm Sel}(K,\Tc)$ and $X(K,\Ac)$ have $\Lambda^\ac$-rank $1$, and
\[
Ch_{\Lambda^\ac}(X(K,\Ac)_{\rm tors})=
Ch_{\Lambda^\ac}\biggl(\frac{{\rm Sel}(K,\Tc)}{\Lambda^\ac\cdot\mathbf{z}_f}\biggr)^2
\]
as ideals in $\Lambda^\ac\otimes_{\bZ_p}\bQ_p$.
\end{conj}

\begin{rem}
When $N^-=1$ (so that $X=X_{N^+,N^-}$ is just the classical modular curve $X_0(N^+)$), an integral form of
Conjecture~\ref{HP-MC} incorporating the Manin constant associated to the modular
parametrization $\Phi_E$ was originally formulated by Perrin-Riou \cite[Conj.~B]{PR-HP}.
In the above form, Conjecture~\ref{HP-MC} is the case $F=\bQ$ of a
conjecture formulated by Howard \cite[p.3]{howard-PhD-II}.
\end{rem}


Thanks to the work of several authors, Conjecture~\ref{HP-MC} is now
known under mild hypotheses. 

\begin{thm}\label{thm:HP-MC}
Let $E/\bQ$ be an elliptic curve of conductor $N$ with good ordinary reduction at $p\geqslant 5$,
and let $K$ be an imaginary quadratic field of discriminant prime to $N$ satisfying hypothesis ${(\rm Heeg)}$.
Assume in addition that:
\begin{itemize}
\item{} $N$ is square-free,
\item{} $N^-\neq 1$,
\item{} $E[p]$ is ramified at every prime $\ell\mid N^-$,
\item{} ${\rm Gal}(\overline{\bQ}/K)\rightarrow{\rm Aut}_{\bZ_p}(T_pE)$ is surjective.
\end{itemize}
Then:
\begin{enumerate}
\item{} Conjecture~\ref{HP-MC} holds.
\item{} $X_{\emptyset,0}(K,\Ac)$ is $\Lambda^\ac$-torsion, and
\[
Ch_{\Lambda^\ac_{\unr}}(X_{\emptyset,0}(K,\Ac)_{\unr})=(\mathscr{L}_\pp^{\tt BDP}(f/K)^2)
\]
as ideals in $\Lambda^\ac_{\unr}$.
\end{enumerate}
\end{thm}

\begin{proof}
This follows by combining the works
of Howard \cite{howard-PhD-I, howard-PhD-II} and Wan \cite{wanIMC},
and the link between the tow provided by (the weight $2$ case of)
the explicit reciprocity law of \cite{cas-hsieh1} and Poitou--Tate duality.
The details are given in the Appendix to this paper.
\end{proof}

\begin{cor}\label{cor:two-varIMC}
Under the hypotheses of Theorem~\ref{thm:HP-MC}, the following hold:
\begin{enumerate}
\item{} $X_{{\rm Gr},0}(K,\mathbf{A})$ is $\Lambda$-torsion,
${\rm Sel}_{{\rm Gr},\emptyset}(K,\mathbf{T})$ has $\Lambda$-rank $1$, and
\[
Ch_{\Lambda}(X_{{\rm Gr},0}(K,\mathbf{A}))=
Ch_\Lambda\bigg(\frac{{\rm Sel}_{{\rm Gr},\emptyset}(K,\mathbf{T})}{\Lambda\cdot\mathcal{BF}}\biggr).
\]
\item{} $X_{\emptyset,0}(K,\mathbf{A})$ is $\Lambda$-torsion, and
\[
Ch_{\Lambda_{\unr}}(X_{\emptyset,0}(K,\mathbf{A})_{\unr})=(\mathscr{L}_\pp(f/K)).
\]
\end{enumerate}
\end{cor}

\begin{proof}
Part (2) is deduced from Theorem~\ref{thm:HP-MC} by the same argument as in
\cite[Thm.~5.1]{cas-wan-SS}; part (1) then follows from Theorem~\ref{thm:2-varIMC}.
The details, which involve many of the ingredients that go into the proof of Theorem~\ref{thm:HP-MC} is
relegated to the Appendix of this paper.
\end{proof}

\subsection{Rubin's height formula}\label{subsec:rubin}

We maintain the notations from $\S\ref{sec:anti-L}$, and continue to
denote by $L_p(f/K,\Sigma^{(1)})$ the image of the $p$-adic $L$-function
$L_p(f/K,\Sigma^{(1)})$ of Theorem~\ref{thm:hida} under the natural projection
\[
\bQ_p\otimes_{\ro}\ro[[H_{p^\infty}]]\longrightarrow\bQ_p\otimes_{\ro}\Lambda.
\]
Let $\gamma^{\cyc}\in\Gamma^\cyc$ be a topological generator, and using the identification
$\Lambda\simeq\Lambda^\ac[[\Gamma^\cyc]]$ expand
\begin{equation}\label{eq:expand}
L_p(f/K,\Sigma^{(1)})=L_{p,0}^{}(f/K)+L_{p,1}^{\rm cyc}(f/K)(\gamma^\cyc-1)+\cdots 
\end{equation}
as a power series in $(\gamma^\cyc-1)$. The `constant term' in this expansion
is just the anticyclotomic projection $L^\ac_{p}(f/K,\Sigma^{(1)})$, which
vanishes by Proposition~\ref{thm:hida-1}. 

In light of the isomorphism $(\ref{eq:Shapiro})$, the Beilinson--Flach class $\mathcal{BF}$ from Theorem~\ref{thm:Col}
can be seen a compatible system of classes
\[
\mathcal{BF}=\varprojlim_{K\subseteq F\subseteq K_\infty}\mathcal{BF}_F
\]
with $\mathcal{BF}_F\in{\rm Sel}_{{\rm Gr},\emptyset}(F,T)$, and where $F$ runs
over the finite extensions of $K$ contained in $K_\infty$. Let $L_n=K_n^\ac K_\infty^{\rm cyc}$, and set
$\mathcal{BF}(L_n):=\varprojlim_{K\subseteq F\subseteq L_n}\mathcal{BF}_F$.

Recall that for any free $\bZ_p$-module $M$ equipped with a linear action of $G_{\bQ_p}$,
and a $p$-adic Lie extension $E_\infty/\bQ_p$, one defines
\[
H^1_{\rm Iw}(E_\infty,M):=\varprojlim_{\bQ_p\subseteq E\subseteq E_\infty}H^1(E,M),
\]
where similarly as before, $E$ runs over the finite extensions of $\bQ_p$ contained in $E_\infty$.

\begin{lem}\label{lem:3.1.1}
For every $n\geqslant 0$ there is an element
\[
\beta_n\in H^1_{\rm Iw}(L_{n,\overline\pp},T)\otimes_{\bZ_p}\bQ_p,
\]
unique modulo the natural image of $H^1_{\rm Iw}(L_{n,\overline{\pp}},\mathscr{F}^+T)\otimes_{}\bQ_p$,
such that
\[
(\gamma^{\cyc}-1)\beta_n={\rm loc}_{\overline{\pp}}(\mathcal{BF}(L_n)).
\]
Furthermore, the natural images of $\beta_n$ in $H^1(K_{n,\overline\pp}^\ac,T)\otimes_{}\bQ_p/H^1_{\rm Gr}(K_{n,\overline\pp}^\ac,T)\otimes_{}\bQ_p$ define a class $\beta_\infty(\mathds{1})\in H^1(K_{\overline\pp},\Tc)\otimes{\bQ_p}/H^1_{\rm Gr}(K_{\overline\pp},\Tc)\otimes\bQ_p$, 
and the $\Lambda$-linear map ${\rm Col}^{(1)}$ of Theorem~\ref{thm:Col}
yields an identification
\[
\frac{H^1(K_{\overline{\pp}},\Tc)}{H^1_{\rm Gr}(K_{\overline\pp},\Tc)}\otimes_{\ro}\bQ_p
\simeq\Lambda^\ac\otimes_{\ro}\bQ_p
\]
sending $\beta_\infty(\mathds{1})$ to the `linear term' $L_{p,1}^{\cyc}(f/K)$
in $(\ref{eq:expand})$.
\end{lem}

\begin{proof}
By the injectivity of ${\rm Col}^{(1)}$,
the first claim follows easily from the explicit reciprocity law of Theorem~\ref{thm:Col} and the vanishing
of $L_{p,0}(f/K)$. The second claim is a direct consequence
of the definitions of $\beta_\infty(\mathds{1})$ and $L^{\cyc}_{p,1}(f/K)$
(and again Theorem~\ref{thm:Col}).
\end{proof}

Recall that by Corollary~\ref{cor:str-Gr}
we have 
the inclusion $\mathcal{BF}_n^\ac\in{\rm Sel}(K_n^\ac,T)$ for every $n\geqslant 0$. 
Let $\cI$ be the augmentation ideal of $\ro[[\Gamma^\cyc]]$,
and set $\mathcal{J}=\cI/\cI^2$.

\begin{thm}\label{thm:rubin-ht}
For every $n\geqslant 0$ there is a 
$p$-adic height pairing
\[
\langle\;,\;\rangle_{K_n^\ac}^{\cyc}:{\rm Sel}_{\Gr}(K_n^\ac,T)\times{\rm Sel}_{\Gr}(K_n^\ac,T)
\longrightarrow p^{-k}\bZ_p\otimes_{\bZ_p}\mathcal{J}
\]
for some $k\in\bZ_{\geqslant 0}$ independent of $n$,
such that for every $b\in{\rm Sel}_{\Gr}(K_n^\ac,T)$, we have
\begin{equation}\label{eq:rubin-ht}
\langle\mathcal{BF}^\ac_{n},b\rangle^{\rm cyc}_{K_n^\ac}=(\beta_n(\mathds{1}),{\rm loc}_{\overline\pp}(b))_n\otimes(\gamma^\cyc-1),
\end{equation}
where $(\;,\;)_n$ is the $\bQ_p$-linear extension of the local Tate pairing
\[
\frac{H^1(K_{n,\overline\pp}^\ac,T)}{H^1_{\rm Gr}(K_{n,\overline\pp}^\ac,T)}
\times H^1_{\rm Gr}(K_{n,\overline\pp}^\ac,T)\longrightarrow\ro.
\]
\end{thm}

\begin{proof}
The construction of the $p$-adic height pairing follows work of Perrin-Riou~\cite{PR-109},
and the proof of the height formula $(\ref{eq:rubin-ht})$ follows from a straightforward
extension of a well-known result due to Rubin \cite{rubin-ht}
(cf. \cite[\S{3.2}]{AHsplit} and \cite[(13.3.14)]{nekovar310}).
\end{proof}

\subsection{Main results}

Define the $\Lambda^\ac$-adic height pairing
\begin{equation}\label{eq:lambda-ht}
\langle\;,\;\rangle_{K^\ac_\infty}^{\cyc}:{\rm Sel}_{\Gr}(K,\Tc)
\otimes_{\Lambda^\ac}{\rm Sel}_{\Gr}(K,\Tc)^\iota
\longrightarrow\bQ_p\otimes_{\bZ_p}\Lambda^\ac\otimes_{\bZ_p}\mathcal{J}
\end{equation}
by the formula
\[
\langle a_\infty,b_\infty\rangle^{\rm cyc}_{K_\infty^\ac}=
\varprojlim_n\sum_{\sigma\in{\rm Gal}(K_n^\ac/K)}\langle a_n,b_n^\sigma\rangle^{\rm cyc}_{K_n^\ac}\cdot\sigma,
\]
and define the cyclotomic regulator $\mathcal{R}_{\rm cyc}\subseteq\Lambda^\ac$
to be the characteristic ideal of the cokernel of $(\ref{eq:lambda-ht})$.
The following $\Lambda^\ac$-adic Birch and Swinnerton-Dyer formula
corresponds to Theorem~B in the Introduction.

\begin{thm}\label{thm:3.1.5}
Let the hypotheses be as in Theorem~\ref{thm:HP-MC}, and denote by
$\mathcal{X}^{}_{\rm tors}$ the characteristic ideal of 
$X_{\Gr}(K,\Ac)_{\rm tors}$. Then
\[
\mathcal{R}_{\rm cyc}^{}\cdot\mathcal{X}_{\rm tors}
=(L_{p,1}^{\cyc}(f/K))
\]
as ideals in $\Lambda^\ac\otimes_{\bZ_p}\bQ_p$.
\end{thm}

\begin{proof}
The height formula of Theorem~\ref{thm:rubin-ht} and Lemma~\ref{lem:3.1.1} immediately
yield the equality
\begin{equation}\label{eq:cor-ht}
\mathcal{R}_{\rm cyc}\cdot
Ch_{\Lambda^\ac}\biggl(\frac{{\rm Sel}_{\Gr}(K,\Tc)}{\Lambda^\ac\cdot\mathcal{BF}^\ac}\biggr)
=(L_{p,1}^{\cyc}(f/K))\cdot\eta^\iota,
\end{equation}
where $\eta\subseteq\Lambda^\ac$ is the characteristic ideal of
$H^1_{\rm Gr}(K_{\overline\pp},\Tc)/{\rm loc}_{\overline\pp}({\rm Sel}_{\Gr}(K,\Tc))$. 
By Theorem~\ref{thm:HP-MC} we have
${\rm rank}_{\Lambda^\ac}({\rm Sel}_{\Gr}(K,\Tc))=1$, and so ${\rm Sel}_{0,{\rm Gr}}(L,\Tc)=\{0\}$
by Lemma~\ref{lem:cas}, from where the nonvanishing of $\eta$ follows
by the exactness of (\ref{eq:tauto-es}) in the proof of that lemma.

Now, from global duality we have the exact sequence
\begin{equation}\label{eq:PT2}
0\longrightarrow\frac{H^1_{\rm Gr}(K_{\overline\pp},\Tc)}
{{\rm loc}_{\overline\pp}({\rm Sel}_{\Gr}(K,\Tc))}
\longrightarrow X_{{\rm Gr},\emptyset}(K,\Ac)
\longrightarrow X_{\Gr}(K,\Ac)\longrightarrow 0.
\end{equation}
Since ${\rm Sel}_{\Gr}(K,\Tc)$ has $\Lambda^\ac$-rank $1$, by Lemma~\ref{thm:ES}
we know that $X_{{\rm Gr},0}(K,\Ac)$ is $\Lambda^\ac$-torsion
and that
\begin{equation}\label{eq:str-to-sel}
{\rm Sel}(K,\Tc)={\rm Sel}_{{\rm Gr},\emptyset}(K,\Tc).
\end{equation}
The discussion above also shows that the second and third terms in $(\ref{eq:PT2})$
both have $\Lambda^\ac$-rank $1$. Taking $\Lambda^\ac$-torsion in $(\ref{eq:PT2})$
and applying Lemma~\ref{lem:str-rel}, we thus obtain the equality
\begin{equation}\label{eq:str-rel}
Ch_{\Lambda^\ac}(X_{{\rm Gr},0}(K,\Ac))=\mathcal{X}_{\rm tors}\cdot\eta^\iota
\end{equation}
as ideals in $\Lambda^\ac\otimes_{\bZ_p}\bQ_p$.
By Proposition~\ref{thm:str} (whose conclusion can be invoked thanks to Corollary~\ref{cor:two-varIMC}),
equality $(\ref{eq:str-rel})$ amounts to 
\begin{equation}\label{eq:str-rel-2}
Ch_{\Lambda^\ac}\biggl(\frac{{\rm Sel}_{\Gr}(K,\Tc)}{\Lambda^\ac\cdot\mathcal{BF}^\ac}\biggr)
=\mathcal{X}_{\rm tors}\cdot\eta^\iota,
\end{equation}
using equality $(\ref{eq:str-to-sel})$ in the left-hand side.
Substituting $(\ref{eq:str-rel-2})$ into $(\ref{eq:cor-ht})$,
the result follows from the nonvanishing of $\eta$.
\end{proof}

The $\Lambda^\ac$-adic Gross--Zagier formula of Theorem~A in the Introduction is
now a direct consequence of Theorem~\ref{thm:3.1.5} and
the Heegner point main conjecture of Theorem~\ref{thm:HP-MC}.

\begin{thm}
Under the hypotheses of Theorem~\ref{thm:HP-MC} we have the equality
\[
(L_{p,1}^{\cyc}(f/K))=(\langle\mathbf{z}_f,\mathbf{z}_f\rangle_{K_\infty^\ac}^{\cyc})
\]
as ideals of $\Lambda^\ac\otimes_{\bZ_p}\bQ_p$.
\end{thm}

\begin{proof}
By Theorem~\ref{thm:HP-MC} we have the equality
\begin{equation}\label{eq:HPMC}
Ch_{\Lambda^\ac}(X_{\Gr}(K,\Ac)_{\tors})
=Ch_{\Lambda^\ac}\biggl(\frac{{\rm Sel}_{\Gr}(K,\Tc)}{\Lambda^\ac\cdot\mathbf{z}_{f}}\biggr)^2\nonumber
\end{equation}
up to powers of $p\Lambda^\ac$. Combined with Theorem~\ref{thm:3.1.5}, we conclude that
\begin{align*}
(L_{p,1}^{\cyc}(f/K))
&=\mathcal{R}_{\cyc}
\cdot Ch_{\Lambda^\ac}\biggl(\frac{{\rm Sel}_{\Gr}(K,\Tc)}{\Lambda^\ac\cdot\mathbf{z}_{f}}\biggr)^2
=(\langle\mathbf{z}_{f},\mathbf{z}_{f}\rangle_{K_\infty^\ac}^{\cyc})
\end{align*}
as ideals in $\Lambda^\ac\otimes_{\bZ_p}\bQ_p$, as was to be shown.
\end{proof}

\begin{appendix}

\section{Proofs of Theorem~\ref{thm:HP-MC} and Corollary~\ref{cor:two-varIMC}}

The purpose of this Appendix is to develop some of the details that go into the proofs
of Theorem~\ref{thm:HP-MC} and Corollary~\ref{cor:two-varIMC}.
As already mentioned, the first result follows easily
from the combination of three ingredients: Howard's divisibility in the 
Heegner point main conjecture, Wan's converse divisibility 
in the main conjecture for $\mathscr{L}_\pp(f/K)$,
and an explicit reciprocity law for Heegner points building a bridge between the two. 
The arguments that follow closely parallel those in \cite[\S{4.3}]{cas-wan-SS},
and the interested reader may also wish to consult \cite{wan}.
\sk

Throughout, we let $E/\bQ$ be an elliptic curve of conductor $N$, $f\in S_2(\Gamma_0(N))$
the normalized newform associated with $E$, $p\geqslant 5$ a prime of good reduction for $E$, and
$K/\bQ$ an imaginary quadratic field of discriminant prime to $Np$ satisfying
hypothesis (Heeg) in the Introduction.
We begin by recalling the third of the aforementioned results.

\begin{thm}\label{thm:cas-hsieh}
Assume in addition that $p=\pp\overline\pp$ splits in $K$ and that $f$ is $p$-ordinary.
Then there exists an injective $\Lambda^\ac$-linear map
\[
\mathcal{L}_{+}:H^1(K_\pp,\mathscr{F}^+\Tc)_{\unr}\longrightarrow\Lambda_{\unr}
\]
with finite cokernel such that
\[
\mathcal{L}_{+}({\rm res}_\pp(\mathbf{z}_f))=-\mathscr{L}_\pp^{\tt BDP}(f/K)\cdot\sigma_{-1,\pp},
\]
where $\sigma_{-1,\pp}\in\Gamma^\ac$ has order two.
\end{thm}

\begin{proof}
This follows from the weight $2$ case of \cite[Thm.~5.7]{cas-hsieh1}. (We note that
the injectivity of the map $\mathcal{L}_+$ is not explicitly stated in \emph{loc.cit.},
but it follows from \cite[Prop.~4.11]{LZ2} and the construction in \cite[Thm.~5.1]{cas-hsieh1}.
Also, since $\mathbf{z}_f\in{\rm Sel}(K,\Tc)$, the restriction ${\rm res}_\pp(\mathbf{z}_f)$
lands in $H^1_{\rm Gr}(K_\pp,\Tc)$, which is naturally identified with $H^1(K_\pp,\mathscr{F}^+\Tc)$
by the vanishing of $H^0(K_\pp,\mathscr{F}^-\Tc)$.)
\end{proof}

The next three lemmas are applications of Poitou--Tate duality, combined with
the explicit reciprocity law of Theorem~\ref{thm:cas-hsieh}, and in all of them
the assumption that $p=\pp\overline{\pp}$ splits in $K$ is in order.

\begin{lem}\label{thm:ES}
Assume that ${\rm Sel}(K,\Tc)$ has $\Lambda^\ac$-rank $1$. Then
\begin{equation}\label{eq:pm-rel}
{\rm Sel}(K,\Tc)={\rm Sel}_{{\rm Gr},\emptyset}(K,\Tc)
\end{equation}
and $X_{{\rm Gr},0}(K,\Ac)$ is $\Lambda^\ac$-torsion.
\end{lem}

\begin{proof}
Consider the exact sequence
\begin{equation}\label{eq:es0}
{\rm Sel}_{\Gr}(K,\Tc)\overset{{\rm loc}_\pp}\longrightarrow H^1_{\rm Gr}(K_\pp,\Tc)
\longrightarrow X_{\emptyset,{\rm Gr}}(K,\Ac)\longrightarrow X_{}(K,\Ac)\longrightarrow 0.
\end{equation}
Since $\mathbf{z}_f$ lands in ${\rm Sel}(K,\Tc)$, by Theorem~\ref{thm:cas-hsieh} the nonvanishing of $\mathscr{L}^{\tt BDP}_{\pp}(f/K)$
implies that the image of the map ${\rm loc}_\pp$ is not $\Lambda^\ac$-torsion.
Since $H^1_{\rm Gr}(K_\pp,\Tc)$ has $\Lambda^{\rm ac}$-rank $1$, it follows
that ${\rm coker}({\rm loc}_\pp)$ is $\Lambda^{\rm ac}$-torsion, and since
part (1) of Lemma~\ref{lem:str-rel} combined with our assumption implies
that $X_{}(K,\Ac)$ has $\Lambda^{\rm ac}$-rank $1$,
we conclude from (\ref{eq:es0}) that
\begin{equation}\label{eq:rank=1}
{\rm rank}_{\Lambda^{\rm ac}}(X_{\emptyset,{\rm Gr}}(K,\Ac))=1.
\end{equation}
Since $X_{\emptyset,{\rm Gr}}(K,\Ac)\simeq X_{{\rm Gr},\emptyset}(K,\Ac)$ by the action
of complex conjugation, we deduce from (\ref{eq:rank=1}) and part (2) of Lemma~\ref{lem:str-rel}
that $X_{{\rm Gr},0}(K,\Ac)$ is $\Lambda^{\rm ac}$-torsion.
Finally, since the quotient
$H^1(K_{\overline{\pp}},\Tc)/H^1_{\rm Gr}(K_{\overline{\pp}},\Tc)$ has $\Lambda^{\rm ac}$-rank $1$,
counting ranks in the exact sequence
\begin{equation}\label{eq:es1}
\begin{split}
0\longrightarrow{\rm Sel}(K,\Tc)\longrightarrow
{\rm Sel}_{{\rm Gr},\emptyset}(K,\Tc)
&\overset{{\rm loc}_{\overline\pp}}\longrightarrow
\frac{H^1(K_{\overline{\pp}},\Tc)}{H^1_{\rm Gr}(K_{\overline{\pp}},\Tc)}\\
&\longrightarrow X(K,\Ac)\longrightarrow X_{{\rm Gr},0}(K,\Ac)\longrightarrow 0,\nonumber
\end{split}
\end{equation}
we conclude that ${\rm Sel}(K,\Tc)$ and ${\rm Sel}_{{\rm Gr},\emptyset}(K,\Tc)$ have both
$\Lambda^{\rm ac}$-rank $1$,
and since the quotient $H^1(K_{\overline\pp},\Tc)/H^1_{\rm Gr}(K_{\overline\pp},\Tc)$ is also
$\Lambda^{\rm ac}$-torsion-free (in fact, it injects into $H^1(K_{\overline\pp},\mathscr{F}^-\Tc)$ by the vanishing of $H^0(K_{\overline{\pp}},\mathscr{F}^-\Tc)$),
equality (\ref{eq:pm-rel}) follows.
\end{proof}

\begin{lem}\label{lem:cas}
Assume that ${\rm Sel}(K,\Tc)$ has $\Lambda^\ac$-rank $1$. Then
${\rm Sel}_{0,{\rm Gr}}(K,\Tc)=\{0\}$, and
for any height one prime $\mathfrak{P}$ of $\Lambda_{\unr}^{\ac}$ we have
\[
{\rm ord}_{\mathfrak{P}}(\mathscr{L}^{\tt BDP}_{\pp}(f/K))={\rm length}_{\mathfrak{P}}({\rm coker}({\rm loc}_{{\pp}})_{\unr})
+{\rm length}_{\mathfrak{P}}\biggl(\frac{{\rm Sel}(K,\Tc)_{\unr}}{\Lambda_{\unr}^\ac\cdot\mathbf{z}_f}\biggr),
\]
where ${\rm loc}_{\pp}:{\rm Sel}(K,\Tc)
\rightarrow H^1_{\rm Gr}(K_{\pp},\Tc)$ is the natural restriction map.
\end{lem}

\begin{proof}
Denote by $\Lambda^\ac_{\mathfrak{P}}$ the localization of $\Lambda_{\unr}^\ac$
at $\mathfrak{P}$, and consider the tautological exact sequence
\begin{equation}\label{eq:tauto-es}
0\longrightarrow{\rm Sel}_{0,{\rm Gr}}(K,\Tc)\longrightarrow
{\rm Sel}(K,\Tc)
\xrightarrow{{\rm loc}_\pp} H^1_{\rm Gr}(K_\pp,\Tc)
\longrightarrow{\rm coker}({\rm loc}_\pp)\longrightarrow 0.
\end{equation}

By Theorem~\ref{thm:cas-hsieh}, the nonvanishing of $\mathscr{L}^{\tt BDP}_{\pp}(f/K)$
implies that the image of $\mathbf{z}_f\in{\rm Sel}(K,\Tc)$
under the map ${\rm loc}_\pp$ is not $\Lambda^\ac$-torsion.
Since we assume that ${\rm Sel}(K,\Tc)$ has $\Lambda^\ac$-rank $1$,  
this shows that ${\rm Sel}_{0,{\rm Gr}}(K,\Tc)$
is $\Lambda^\ac$-torsion, and so ${\rm Sel}_{0,{\rm Gr}}(K,\Tc)=\{0\}$ by Lemma~\ref{lem:no-tors}.

From (\ref{eq:tauto-es}) we thus deduce the exact sequence
\begin{equation}\label{eq:1}
0\longrightarrow\frac{{\rm Sel}(K,\Tc)}
{\Lambda^\ac\cdot\mathbf{z}_f}
\xrightarrow{{\rm loc}_\pp}\frac{H^1_{\rm Gr}(K_{\pp},\Tc)}
{\Lambda^\ac\cdot{\rm loc}_{\pp}(\mathbf{z}_f)}
\longrightarrow {\rm coker}({\rm loc}_{\pp})\longrightarrow 0,\nonumber
\end{equation}
and since by 
Theorem~\ref{thm:cas-hsieh} we have a $\Lambda_{\mathfrak{P}}^\ac$-module pseudo-isomorphism
\[
\biggl(\frac{H^1_{\rm Gr}(K_\pp,\Tc)}{\Lambda_{\unr}^\ac\cdot{\rm loc}_\pp(\mathbf{z}_f)}\biggr)
\otimes_{\Lambda^\ac_{\unr}}\Lambda_{\mathfrak{P}}^\ac
\overset{\sim}\longrightarrow\frac{\Lambda_{\mathfrak{P}}^\ac}{\Lambda_{\mathfrak{P}}^\ac\cdot\mathscr{L}^{\tt BDP}_\pp(f/K)},
\]
the result follows.
\end{proof}

\begin{lem}\label{lem:tors}
Assume that ${\rm Sel}(K,\Tc)$ has $\Lambda^\ac$-rank $1$.
Then 
$X_{\emptyset,0}(K,\Ac)$ is $\Lambda^\ac$-torsion, and
for any height one prime $\mathfrak{P}$ of $\Lambda^\ac$ 
we have
\[
{\rm length}_{\mathfrak{P}}(X_{\emptyset,0}(K,\Ac))=
{\rm length}_{\mathfrak{P}}(X(K,\Ac)_{\rm tors})+
2\;{\rm length}_{\mathfrak{P}}({\rm coker}({\rm loc}_{\pp})),
\]
where ${\rm loc}_{\pp}:{\rm Sel}(K,\Tc)
\rightarrow H^1_{\rm Gr}(K_{\pp},\Tc)$ is the natural restriction map.
\end{lem}

\begin{proof}
Global duality yields the exact sequence
\begin{equation}\label{PT3}
0\longrightarrow{\rm coker}({\rm loc}_{\pp})\longrightarrow
X_{\emptyset,{\rm Gr}}(K,\Ac)\longrightarrow
X(K,\Ac)\longrightarrow 0.
\end{equation}
As shown in the proof of Lemma~\ref{lem:cas}, the first term in this sequence is $\Lambda^\ac$-torsion;
since by Lemma~\ref{lem:str-rel}(1) the assumption implies that $X(K,\Ac)$
has $\Lambda^\ac$-rank $1$, this shows that the same is true for $X_{\emptyset,{\rm Gr}}(K,\Ac)$,
and by Lemma~\ref{lem:str-rel}(2) it follows that
$X_{0,{\rm Gr}}(K,\Ac)$ is $\Lambda^\ac$-torsion.
Thus taking $\Lambda^\ac$-torsion in (\ref{PT3}) and using Lemma~\ref{lem:str-rel}(3),
it follows that
\begin{equation}\label{pbar}
{\rm length}_{\mathfrak{P}}(X_{0,{\rm Gr}}(K,\Ac))=
{\rm length}_{\mathfrak{P}}(X(K,\Ac)_{{\rm tors}})+
{\rm length}_{\mathfrak{P}}({\rm coker}({\rm loc}_{\pp}))
\end{equation}
for any height one prime $\mathfrak{P}$ of $\Lambda^\ac$. 

Another application of global duality yields the exact sequence
\begin{equation}\label{PT4}
0\longrightarrow{\rm coker}({\rm loc}_{\pp}^{\emptyset})\longrightarrow
X_{\emptyset,0}(K,\Ac)\longrightarrow
X_{{\rm Gr},0}(K,\Ac)\longrightarrow 0,
\end{equation}
where ${\rm loc}^{\emptyset}_{\pp}:{\rm Sel}_{{\rm Gr},\emptyset}(K,\Tc)\rightarrow
H^1_{\rm Gr}(K_\pp,\Tc)$ is the natural restriction map. By Lemma~\ref{thm:ES},
this is the same as the map ${\rm loc}_\pp$ in the statement, and hence
${\rm coker}({\rm loc}^{\emptyset}_{{\pp}})$
is $\Lambda^\ac$-torsion.
Since $X_{{\rm Gr},0}(K,\Ac)$ is $\Lambda^\ac$-torsion by Lemma~\ref{thm:ES},
we conclude from (\ref{PT4}) that $X_{\emptyset,0}(K,\Ac)$ is $\Lambda^\ac$-torsion.
Combining (\ref{PT4}) and (\ref{pbar}), we thus have
\begin{align*}
{\rm length}_{\mathfrak{P}}(X_{\emptyset,0}(K,\Ac))
&={\rm length}_{\mathfrak{P}}(X_{{\rm Gr},0}(K,\Ac))+
{\rm length}_{\mathfrak{P}}({\rm coker}({\rm loc}_{{\pp}})) \\
&={\rm length}_{\mathfrak{P}}(X(K,\Ac)_{\rm tors})+
2\;{\rm length}_{\mathfrak{P}}({\rm coker}({\rm loc}_{\pp}))
\end{align*}
for any height one prime $\mathfrak{P}$ of $\Lambda^\ac$, 
as was to be shown.
\end{proof}

The next two theorems recall the results by Howard and Wan that we need.

\begin{thm}[Howard]\label{thm:KS-argument}
Assume that  $f$ is $p$-ordinary and that ${\rm Gal}(\overline{\bQ}/K)\rightarrow{\rm Aut}_{\bZ_p}(T_pE)$ is surjective.
Then both ${\rm Sel}(K,\Tc)$ and $X(K,\Ac)$ have $\Lambda^\ac$-rank $1$ and
we have the divisibility
\[
Ch_{\Lambda^\ac}(X(K,\Ac)_{\rm tors})\supseteq
Ch_{\Lambda^\ac}\biggl(\frac{{\rm Sel}(K,\Tc)}{\Lambda^\ac\cdot\mathbf{z}_f}\biggr)^2.
\]
\end{thm}

\begin{proof}
This follows from \cite[Thm.~B]{howard-PhD-I}, as extended 
by \cite[Thm.~3.4.2, Thm.~3.4.3]{howard-PhD-II} to the context of Shimura curves attached to possibly
non-split indefinite quaternion algebras, \emph{cf.} \cite[Thm.~2.1]{wan}.
(Note also that by the work of Cornut--Vatsal \cite[Thm.~1.10]{CV-dur}
the assumption in \cite[Thm.~3.4.3]{howard-PhD-II} is now known to hold.)
\end{proof}

\begin{rem}
In \cite{howard-PhD-I} and \cite{howard-PhD-II} it is assumed that $p$ does not
divide the class number of $K$, but similarly as in \cite[\S{4}]{kim-parity}
(see esp. [\emph{loc.cit.}, Prop.~4.19]), this assumption may be easily relaxed.
\end{rem}

\begin{thm}[Wan]\label{thm:wan-div}
Assume that:
\begin{itemize}
\item{} $N$ is square-free,
\item{} $N^-\neq 1$,
\item{} $E[p]\vert_{{\rm Gal}(\overline{\bQ}/K)}$ is absolutely irreducible.
\end{itemize}
Then we have the divisibility
\[
Ch_{\Lambda_{\unr}}(X_{\emptyset,0}(K,\mathbf{T})_{\unr})\subseteq(\mathscr{L}_\pp(f/K))
\]
in $\Lambda_{\unr}\otimes_{\bZ_p}\bQ_p$.
\end{thm}

\begin{proof}
The follows from the main result of \cite{wanIMC}.
\end{proof}

The last ingredient that we need is the following result on the vanishing of the
$\mu$-invariant for the $p$-adic $L$-function $\mathscr{L}^{\tt BDP}_\pp(f/K)$ of Theorem~\ref{thm:bdp}.

\begin{thm}[Burungale, Hsieh]\label{thm:mu-zero}
Assume that:
\begin{itemize}
\item{} $N$ is square-free,
\item{} $E[p]\vert_{{\rm Gal}(\overline{\bQ}/K)}$ is absolutely irreducible,
\item{} $E[p]$ is ramified at every prime $\ell\mid N^-$.
\end{itemize}
Then $\mu(\mathscr{L}^{\tt BDP}_\pp(f/K))=0$.
\end{thm}

\begin{proof}
This follows from \cite[Thm.~B]{hsieh}; or alternatively \cite[Thm.~B]{burungale-II},
noting that by the discussion in \cite[p.~912]{prasanna}
our last assumption guarantees that the ratio of Petersson norms
$\alpha(f,f_B)$ in \cite[Thm.~5.6]{burungale-II} is a $p$-adic unit.
\end{proof}

Now we finally assemble all the pieces together.

\begin{proof}[Proof of Theorem~\ref{thm:HP-MC}]
By Theorem~\ref{thm:KS-argument}, ${\rm Sel}(K,\Tc)$ has $\Lambda^\ac$-rank $1$,
and so
$X_{\emptyset,0}(K,\Ac)$ is $\Lambda^\ac$-torsion by Lemma~\ref{lem:tors}.
Let $\mathfrak{P}$ 
be a height one prime of $\Lambda_{\unr}^\ac$,
and set $\mathfrak{P}_0=\mathfrak{P}\cap\Lambda^\ac$. By the divisibility in
Theorem~\ref{thm:KS-argument} we have
\begin{equation}\label{eq:upper}
{\rm length}_{\mathfrak{P}_0}(X(K,\Ac)_{\rm tors})
\leqslant 2\;{\rm length}_{\mathfrak{P}_0}
\biggl(\frac{{\rm Sel}(K,\Ac)}{\Lambda^\ac\cdot\mathbf{z}_f}\biggr).
\end{equation}
Combined with Lemma~\ref{lem:tors} and Lemma~\ref{lem:cas}, respectively, this implies that
\begin{align*}
{\rm length}_{\mathfrak{P}}(X_\pp(K,\Ac)_{\unr})
&\leqslant 2\;{\rm length}_{\mathfrak{P}}
\biggl(\frac{{\rm Sel}(K,\Tc)_{\unr}}{\Lambda_{\unr}^\ac\cdot\mathbf{z}_f}\biggr)+
2\;{\rm length}_{\mathfrak{P}}({\rm coker}({\rm loc}_{\pp})_{\unr})\\
&=2\;{\rm ord}_{\mathfrak{P}}(\mathscr{L}^{\tt BDP}_{\pp}(f/K)),
\end{align*}
and hence
\begin{equation}\label{eq:div}
Ch_{\Lambda_{\unr}^\ac}(X_{\emptyset,0}(K,\Ac))\;\supseteq\;
(\mathscr{L}_{\pp}^{\tt BDP}(f/K)^2)
\end{equation}
and ideals in $\Lambda_{\unr}^\ac$.
Now let $I^{\rm cyc}$ be the principal ideal of $\Lambda$ generated by $\gamma^{\rm cyc}-1$.
By a standard control theorem (\emph{cf.} \cite[Prop.~3.9]{SU}, or see \cite[\S{3.3}]{JSW} for a proof in our same context),
the natural restriction map $H^1(K_\infty^\ac,E[p^\infty])\rightarrow H^1(K_\infty,E[p^\infty])$
induces an isomorphism
\begin{equation}\label{eq:control}
X_{\emptyset,0}(K,\mathbf{A})/I^{\rm cyc}X_{\emptyset,0}(K,\mathbf{A})
\simeq X_{\emptyset,0}(K,\Ac) \nonumber
\end{equation}
as $\Lambda^\ac$-modules. Combined with Corollary~\ref{cor:wan-bdp},
the divisibility in Theorem~\ref{thm:wan-div} thus yields the divisibility
\[
Ch_{\Lambda_{\unr}^\ac}(X_{\emptyset,0}(K,\Ac)_{\unr})\subseteq(\mathscr{L}^{\tt BDP}_\pp(f/K)^2)
\]
in $\Lambda_{\unr}^\ac\otimes_{\bZ_p}\bQ_p$.
In particular, equality holds in (\ref{eq:div}) up to powers of $p\Lambda^\ac_{\unr}$,
and by Theorem~\ref{thm:mu-zero} the equality holds integrally.
Combined with Lemmas~\ref{lem:cas} and \ref{lem:tors},
this implies that for any height one prime
$\mathfrak{P}_0$ of $\Lambda^\ac$ 
equality in (\ref{eq:upper}) holds, thus concluding the proof.
\end{proof}

\begin{proof}[Proof of Corollary~\ref{cor:two-varIMC}]
As before, let $I^{\rm cyc}\subseteq\Lambda$ be the ideal generated by $\gamma^{\rm cyc}-1$, and
set
\[
X:=Ch_{\Lambda_{\unr}}(X_{\emptyset,0}(K,\mathbf{A})_{\unr}),\quad
Y:=(\mathscr{L}_\pp(f/K))\subseteq\Lambda_{\unr}.
\]
As in the proof of Theorem~\ref{thm:HP-MC}, the divisibility in Theorem~\ref{thm:wan-div}
combined with Theorem~\ref{thm:mu-zero} yields the divisibility $X\subseteq Y$ as ideals in $\Lambda_{\unr}$.
On the other hand, from the combination of (\ref{eq:control}) and Corollary~\ref{cor:wan-bdp},
Theorem~\ref{thm:HP-MC} implies that $X_{\emptyset,0}(K,\mathbf{A})$ is $\Lambda$-torsion and that
$X=Y\pmod{I^{\rm cyc}}$. The equality $X=Y$ in $\Lambda_{\unr}$
thus follows from \cite[Lem.~3.2]{SU}. This completes the proof of part (2) of Corollary~\ref{cor:two-varIMC}.
By Theorem~\ref{thm:2-varIMC},
it follows that the equality in
part (1) holds in $\Lambda[1/P]$. Since by Theorem~\ref{thm:Col}
 the Beilinson--Flach class $\mathcal{BF}$ restricts nontrivially to the cyclotomic line (see Remark~\ref{rem:non0}),
 the aforementioned equality holds in $\Lambda$, concluding the proof of the result.
\end{proof}

\end{appendix}

\bibliographystyle{amsalpha}
\bibliography{Heegner}

\end{document}